\documentclass[10pt,a4paper,twoside]{article}

\usepackage{etoolbox}
\csundef{leftbar}
\csundef{endleftbar}

\usepackage[utf8]{inputenc}
\usepackage[T1]{fontenc}
\usepackage[french,english]{babel} 
\usepackage{lmodern}
\usepackage{makeidx} 
\usepackage{graphicx}

\usepackage{caption}
\usepackage{eurosym}
\usepackage{xfrac}
\usepackage{mathrsfs}
\usepackage{stmaryrd}
\usepackage{eqnarray}
\usepackage{multirow}
\usepackage{hyperref}
\usepackage{bold-extra}
\usepackage{blkarray}

\usepackage{amsmath}
\usepackage{amsthm}
\usepackage{thmtools}
\declaretheoremstyle[
	spaceabove=\topsep, spacebelow=\topsep,
	headfont=\normalfont\bfseries,
	notefont=\normalfont, notebraces={(}{)},
	bodyfont=\normalfont,
	postheadspace=.5em,
]{defn}
\declaretheorem[numberwithin=section,name=Definition,style=defn]{defi}
\declaretheorem[sibling=defi,name=Example,style=defn]{exe}

\declaretheorem[numbered=no,name=Conjecture,style=defn]{conj}

\declaretheorem[sibling=defi,name=Conjecture,style=defn]{con}
\declaretheoremstyle[
	spaceabove=\topsep, spacebelow=\topsep,
	headfont=\normalfont\bfseries,
	notefont=\normalfont, notebraces={(}{)},
	bodyfont=\itshape,
	postheadspace=.5em,
	qed=
]{proposition}
\declaretheorem[sibling=defi,name=Proposition,style=proposition]{prop}

\declaretheorem[sibling=defi,name=Theorem,style=proposition]{theo}

\declaretheorem[sibling=defi,name=Lemma,style=proposition]{lem}
\declaretheorem[sibling=defi,name=Corollary,style=proposition]{cor}
\declaretheoremstyle[
	spaceabove=\topsep, spacebelow=\topsep,
	headfont=\normalfont\itshape,
	notefont=\normalfont, notebraces={(}{)},
	bodyfont=\normalfont,
	postheadspace=.5em,
	qed=
]{remarque}
\declaretheorem[sibling=defi,name=Remark,style=remarque]{rema}

\usepackage{amssymb}
\usepackage{amstext}

\makeatletter
\DeclareFontFamily{OMX}{MnSymbolE}{}
\DeclareSymbolFont{MnLargeSymbols}{OMX}{MnSymbolE}{m}{n}
\SetSymbolFont{MnLargeSymbols}{bold}{OMX}{MnSymbolE}{b}{n}
\DeclareFontShape{OMX}{MnSymbolE}{m}{n}{
    <-6>  MnSymbolE5
   <6-7>  MnSymbolE6
   <7-8>  MnSymbolE7
   <8-9>  MnSymbolE8
   <9-10> MnSymbolE9
  <10-12> MnSymbolE10
  <12->   MnSymbolE12
}{}
\DeclareFontShape{OMX}{MnSymbolE}{b}{n}{
    <-6>  MnSymbolE-Bold5
   <6-7>  MnSymbolE-Bold6
   <7-8>  MnSymbolE-Bold7
   <8-9>  MnSymbolE-Bold8
   <9-10> MnSymbolE-Bold9
  <10-12> MnSymbolE-Bold10
  <12->   MnSymbolE-Bold12
}{}

\let\llangle\@undefined
\let\rrangle\@undefined
\DeclareMathDelimiter{\llangle}{\mathopen}%
                     {MnLargeSymbols}{'164}{MnLargeSymbols}{'164}
\DeclareMathDelimiter{\rrangle}{\mathclose}%
                     {MnLargeSymbols}{'171}{MnLargeSymbols}{'171}
\makeatother

\usepackage{tikz}
\usetikzlibrary{arrows,matrix,cd}

\usepackage{bbm}
\usepackage{multicol}

\usepackage{import}

\usepackage{lettrine}

  {\vspace*{\fill}\clearpage\restoregeometry}
  
\usepackage[french,lined]{algorithm2e}

\newcommand{\Zb}{\mathbb{Z}}
\newcommand{\Rb}{\mathbb{R}}

\newcommand{\Cb}{\mathbb{C}}

\newcommand{\Abb}{\mathbb{A}}

\newcommand{\Er}{\mathrm{E}}

\newcommand{\m}{\mathfrak{m}}

\newcommand{\Kr}{\mathrm{K}}

\newcommand{\MW}{\mathrm{MW}}

\newcommand{\Nis}{\mathrm{Nis}}

\newcommand{\Sr}{\mathrm{S}}

\newcommand{\CH}{\mathrm{CH}}
\newcommand{\Hr}{\mathrm{H}}
\newcommand{\RS}{\mathrm{RS}}

\newcommand{\Cr}{\mathrm{C}}

\newcommand{\Kbf}{\mathbf{K}}
\newcommand{\Ibf}{\mathbf{I}}
\newcommand{\Ir}{\mathrm{I}}

\newcommand{\Fr}{\mathrm{F}}

\newcommand{\Wr}{\mathrm{W}}
\newcommand{\et}{\mathrm{ét}}

\newcommand{\Cs}{\mathsf{C}}
\newcommand{\Rr}{\mathrm{R}}

\newcommand{\Qr}{\mathrm{Q}}
\newcommand{\Mr}{\mathrm{M}}
\newcommand{\Wbf}{\mathbf{W}}
\newcommand{\Gr}{\mathrm{G}}

\newcommand{\Jbf}{\mathbf{J}}
\newcommand{\Qbf}{\mathbf{Q}}

\newcommand{\BH}{\mathrm{BH}}

\newcommand{\Abf}{\mathbf{A}}
\newcommand{\Bbf}{\mathbf{B}}

\newcommand{\alg}{\mathrm{alg}}
\newcommand{\Osc}{\mathscr{O}}
\newcommand{\Esc}{\mathscr{E}}

\newcommand{\Sm}{\mathrm{Sm}}
\newcommand{\Sch}{\mathrm{Sch}}
\newcommand{\Lc}{\mathcal{L}}
\newcommand{\Mbf}{\mathbf{M}}
\newcommand{\Fbf}{\mathbf{F}}
\newcommand{\Gm}{\mathbb{G}_m}
\newcommand{\Fs}{\mathsf{F}}
\newcommand{\Abr}{\mathrm{Ab}}

\newcommand{\Hbf}{\mathbf{H}}

\newcommand{\Asc}{\mathscr{A}}
\newcommand{\Hsc}{\mathscr{H}}
\newcommand{\Isc}{\mathscr{I}}

\newcommand{\Nbf}{\mathbf{N}}

\newcommand{\Zar}{\mathrm{Zar}}

\newcommand{\Cc}{\mathcal{C}}

\newcommand{\Csc}{\mathscr{C}}

\DeclareMathOperator{\Ker}{Ker}

\DeclareMathOperator{\Id}{Id}

\DeclareMathOperator{\Hom}{Hom}

\DeclareMathOperator{\Spec}{Spec}

\DeclareMathOperator{\Frac}{Frac}

\DeclareMathOperator{\Coker}{Coker}
\DeclareMathOperator{\colim}{colim}

\DeclareMathOperator{\sign}{sign}

\DeclareMathOperator{\Pic}{Pic}
\DeclareMathOperator{\Hdg}{Hdg}

\DeclareMathOperator{\Sq}{Sq}

\renewcommand{\Im}{\operatorname{Im}}

\renewcommand{\H}{\mathrm{H}}

\makeatletter 

\newcommand*\Neginternal[3]{\mathpalette\Neg@{{#1}{#2}{#3}}}
\newcommand*\Neg@[2]{\Neg@@{#1}#2}
\newcommand*\Neg@@[4]{%
  \mathrel{\ooalign{%
    $\m@th#1#4$\cr
    \hidewidth$\m@th#3{#1}\mkern\muexpr#2*2$\hidewidth\cr
  }}%
}

\newcommand*\negslash[1]{\m@th#1\not\mathrel{\phantom{=}}}
\newcommand*\snegslash[1]{\rotatebox[origin=c]{60}{$\m@th#1-$}}
\newcommand*\ssnegslash[1]{\rotatebox[origin=c]{60}{$\m@th#1{\dabar@}\mkern-7mu{\dabar@}$}}
\newcommand*\sssnegslash[1]{\rotatebox[origin=c]{60}{$\m@th#1\dabar@$}}
\makeatother  

\usepackage[newparttoc,pagestyles,toctitles]{titlesec}
\usepackage{titletoc}

\renewcommand*\thesection{\arabic{section}}
\titleformat{\section}[block]{\Large\scshape\filcenter}{\thesection.}{12pt}{}[]

\newpagestyle{main}{

\sethead[][\thepage][]
	{}{\thepage}{}
\setfoot[][][]
	{}{}{}}
	
\usepackage{minitoc}	
\mtcselectlanguage{english}	
\doparttoc

\usepackage{appendix}

 \usepackage[style=english]{csquotes}
 
 \newsavebox{\pullback}
\sbox\pullback{%
\begin{tikzpicture}%
\draw (0,0) -- (1ex,0ex);%
\draw (1ex,0ex) -- (1ex,1ex);%
\end{tikzpicture}}

\usepackage[backend=biber,style=ext-alphabetic]{biblatex}
\DefineBibliographyExtras{english}{\DeclarePunctuationPairs{comma}{*?!}}

\DeclareBibliographyDriver{article}{%
  \usebibmacro{bibindex}%
  \usebibmacro{begentry}%
  \usebibmacro{author}%
  \setunit{\addcomma\space}
  \printfield[emph]{title}
  \setunit{\addcomma\space}
  \printfield{journaltitle}
  \setunit{\addspace}%
  \textbf{\printfield{volume}}
  \setunit{\addspace}%
  \printfield[parens]{year}
  \setunit{\addcomma\space}
  \printfield{number}
  \setunit{\addcomma\space}
  \printfield{pages}
  \usebibmacro{finentry}%
}

\DeclareFieldFormat[article]{title}{#1}
\DeclareFieldFormat[article]{journaltitle}{#1}
\DeclareFieldFormat[article]{pages}{#1}
\DeclareFieldFormat[article]{number}{\bibstring{number}~#1}

\DeclareBibliographyDriver{book}{%
  \usebibmacro{bibindex}%
  \usebibmacro{begentry}%
  \printnames{author}
  \setunit{\addcomma\space}
  \printfield{title}
  \setunit{\addcomma\space}
  \printfield{edition}
  \setunit{\addcomma\space}
  \printfield{series}
  \setunit{\addcomma\space}
  \printfield{volume}
  \setunit{\addcomma\space}
  \printfield{number}
  \setunit{\addcomma\space}
  \printlist{publisher}
  \setunit{\addcomma\space}
  \printlist{location}
  \setunit{\addcomma\space}
  \printfield{year}
  \usebibmacro{finentry}%
}

\DeclareFieldFormat[book]{volume}{\bibstring{volume}~#1}
\DeclareFieldFormat[book]{number}{\bibstring{number}~#1}
\DeclareFieldFormat[book]{publisher}{#1}

\DeclareBibliographyDriver{incollection}{%
  \usebibmacro{bibindex}%
  \usebibmacro{begentry}%
  \printnames{author}
  \setunit{\addcomma\space}
  \printfield{title}
  \setunit{\addperiod\space}
  \usebibmacro{in:}
  \printfield[emph]{booktitle}
  \setunit{\addcomma\space}
  \usebibmacro{byeditor+others}
  \setunit{\addcomma\space}
  \printfield{edition}
  \setunit{\addcomma\space}
  \printfield{series}
  \setunit{\addcomma\space}
  \printfield{volume}
  \setunit{\addcomma\space}
  \printfield{number}
  \setunit{\addcomma\space}
  \printlist{publisher}
  \setunit{\addcomma\space}
  \printlist{location}
  \setunit{\addcomma\space}
  \printfield{year}
  \setunit{\addcomma\space}
  \printfield{pages}
  \usebibmacro{finentry}%
}

\DeclareFieldFormat[incollection]{volume}{\bibstring{volume}~#1}
\DeclareFieldFormat[incollection]{number}{\bibstring{number}~#1}
\DeclareFieldFormat[incollection]{series}{#1,\addspace}

\DefineBibliographyExtras{english}{%
}%

\usepackage{authblk}

\usepackage[top=2cm, bottom=2cm, left=2cm, right=2cm]{geometry}

\numberwithin{equation}{section}

\begin{document}

\pagestyle{main}

\title{On the image of higher signature maps}
\author{Samuel Lerbet}
\affil{DMA, École normale supérieure, Université PSL, CNRS, 75005 Paris, France}
\date{\today}

\maketitle

\begin{abstract}
Given a smooth variety $X$ over the field $\mathbb{R}$ of real numbers and a line bundle $\mathcal{L}$ on $X$ with associated topological line bundle $L=\mathcal{L}(\mathbb{R})$, we study the quadratic real cycle class map $\widetilde{\gamma}_{\mathbb{R}}^c:\widetilde{\mathrm{CH}}^c(X,\mathcal{L})\rightarrow\mathrm{H}^c(X(\mathbb{R}),\mathbb{Z}(L))$ from the $c$-th Chow-Witt group of $X$ to the $c$-th cohomology group of its real locus $X(\mathbb{R})$ with coefficients in the local system $\mathbb{Z}(L)$ associated with $L$. We focus on the cases $c\in\{0,d-2,d-1,d\}$ where $d$ is the dimension of $X$ and we formulate a precise conjecture on the image of $\widetilde{\gamma}_{\mathbb{R}}$ in terms of the exponents of its cokernel that is corroborated by the results obtained in those codimensions.
\end{abstract}

\section{Introduction}

\paragraph*{The complex cycle class map.} Let $X$ be a smooth complex algebraic variety (by convention, an algebraic variety over a field $k$ is a separated $k$-scheme of finite type). The \emph{complex cycle class map} in codimension $c$ is a group homomorphism \[\gamma_\Cb^c(X):\CH^c(X)\rightarrow\Hr^{2c}(X(\Cb),\Zb)\] from the Chow group of codimension $c$ cycles on $X$ to the integral singular cohomology of $X(\Cb)$ in degree $2c$. The Hodge conjecture predicts the image of $\gamma_\Cb^c(X)$ when $X$ is projective. The map $\gamma_\Cb^c(X)$ is defined as follows. Let $i_Y:Y\hookrightarrow X$ be a subvariety of $X$ of codimension $c$. Choose a resolution of singularities $p:\widetilde{Y}\rightarrow Y$ of $Y$. The map $i_Y\circ p(\Cb):\widetilde{Y}(\Cb)\rightarrow X(\Cb)$ induced on complex loci is a proper map of complex analytic manifolds which are naturally oriented as smooth manifolds. Therefore it induces a pushforward homomorphism $i_Y\circ p(\Cb)_*:\Hr^{*}(\widetilde{Y}(\Cb),\Zb)\rightarrow\Hr^{*+2c}(X(\Cb),\Zb)$ in integral singular cohomology and we set $\gamma_\Cb^c(X)[Y]=i_Y\circ p(\Cb)_*(1)$ where $1\in\Hr^0(\widetilde{Y}(\Cb),\Zb)=\Zb$. If $Y$ is smooth, then this is exactly the class of the submanifold $Y(\Cb)$ of $X(\Cb)$ in the cohomology of $X(\Cb)$ hence the name of the map $\gamma_\Cb$.

Assume that $X$ is projective. Then the elements of the image of $\gamma_\Cb^c(X)$ are not arbitrary. Indeed, complex analysis on the compact analytic manifold $X(\Cb)$ may be used to identify a subgroup $\Hdg^{2c}(X(\Cb),\Zb)$ of $\Hr^{2c}(X(\Cb),\Zb)$ of \emph{$(c,c)$-Hodge classes} through which $\gamma_\Cb^c(X)$ factors. Then the variety $X$ is said to satisfy the integral Hodge conjecture in codimension $c$ if $\gamma_\Cb^c(X):\CH^c(X)\rightarrow\Hdg^{2c}(X(\Cb),\Zb)$ is surjective. It is possible for this property to fail to be satisfied: the first examples date back to Atiyah and Hirzebruch \cite{atiyahAnalyticCyclesComplex1962a} who constructed classes that are torsion, and are thus automatically $(c,c)$-Hodge classes, but that do not lie in the image of the complex cycle class map. Nevertheless, in view of the considerable impact of the Hodge conjecture on the development of complex algebraic geometry over the last eighty years, the \emph{rational} Hodge conjecture being still open, it seems natural to define and study counterparts of this conjecture in other areas of algebraic geometry.

\paragraph*{The Borel--Haefliger cycle class map.} In this paper, we are interested in \emph{real} algebraic geometry. In this context, it is well-known that one cannot directly copy the construction described previously to define a cycle class map with values in the \emph{integral} cohomology of the real locus $X(\Rb)$ of a smooth variety $X$ over $\Rb$. This is because the smooth manifold $X(\Rb)$ does not necessarily admit a $\Zb$-valued orientation; nor do the real loci of resolutions of singularities of the subvarieties of $X$, and resolutions of singularities need not be compatible with orientations if they exist. Thus we do not have a proper pushforward map in \emph{integral} cohomology of the real loci in this setting. However, since every smooth manifold has exactly one $\Zb/2$-valued orientation, one can define as above a \emph{mod $2$} cycle class map called the Borel--Haefliger cycle class map \[\gamma_\BH^c(X):\CH^c(X)\rightarrow\Hr^c(X(\Rb),\Zb/2)\] that was first constructed in \cite{borelClasseDhomologieFondamentale1961}. The image of $\gamma_\BH^c(X)$, denoted by $\Hr_\alg^c(X(\Rb),\Zb/2)$, has been abundantly studied, for it is closely related to classical problems in real algebraic geometry. As an example of such a problem, we can cite the approximation of submanifolds of $X(\Rb)$ by real loci of subvarieties of $X$. Nevertheless, the information carried by mod $2$ cohomology is too coarse for one to be satisfied by the Borel--Haefliger cycle class map $\gamma_\BH$ as an analogue of the complex cycle class map $\gamma_\Cb$. Any remedy to this observation should however be compatible with $\gamma_\BH$ in a suitable sense in view of its significance. To the author's knowledge, there are essentially two analogues of the complex cycle class map in real algebraic geometry, relying on entirely different ideas.

\paragraph*{The equivariant cycle class map $\gamma_G$.} This map was first constructed by Krasnov in \cite{krasnovCHARACTERISTICCLASSESVECTOR1992}; this construction is described in detail in \cite{benoistIntegralHodgeConjecture2020}. The idea here is to replace $X(\Rb)$ with $X(\Cb)$ together with the antiholomorphic involution $\sigma:X(\Cb)\rightarrow X(\Cb)$ induced by complex conjugation, which recovers $X(\Rb)$ as the set of fixed points of $\sigma$. One views $\sigma$ as an action of the Galois group $G$ of the extension $\Cb/\Rb$; techniques from equivariant homotopy theory then become available. More precisely, there is an equivariant cycle class map \[\gamma_G^c(X):\CH^c(X)\rightarrow\Hr_\Gr^{2c}(X(\Cb),\Zb(c))\] for every smooth projective real algebraic variety $X$, where $\Hr_\Gr^*$ denotes Borel equivariant cohomology and $\Zb(c)$ is the $G$-module $\Zb(c)=\sqrt{-1}^c\Zb$ with the action of $G$ induced by complex conjugation. This map refines the Borel--Haefliger cycle class map $\gamma_\BH^c(X)$ in the sense that there exists a homomorphism $\pi:\Hr_G^{2c}(X(\Cb),\Zb(c))\rightarrow\Hr^c(X(\Rb),\Zb/2)$ fitting in a commutative triangle of the form
\begin{center}
\begin{tikzcd}
\CH^c(X) \arrow[r,"\protect{\gamma_G^c(X)}"] \arrow[rd,swap,"\protect{\gamma_\BH^c(X)}"] & \Hr_G^{2c}(X(\Cb),\Zb(c)) \arrow[d,"\pi"] \\
                                                                                         & \Hr^c(X(\Rb),\Zb/2)
\end{tikzcd}
\end{center}
(this follows from the case $i=0$ of \cite[Theorem 1.18]{benoistIntegralHodgeConjecture2020}). As in the complex case, the elements of $\Im\gamma_G^c(X)$ are not arbitrary: they satisfy Hodge-theoretic and topological constraints that together determine a subgroup $\Hdg_G^{2c}(X(\Cb),\Zb(c))_0$ of $\Hr_G^{2c}(X(\Cb),\Zb(c))$. Benoist and Wittenberg then say that the real integral Hodge conjecture for codimension $c$ cycles on $X$ holds if the homomorphism $\gamma_G^c(X):\CH^c(X)\rightarrow\Hdg_G^{2c}(X(\Cb),\Zb(c))_0$ is surjective \cite[Definition 2.2]{benoistIntegralHodgeConjecture2020}.

As in the complex case, not every smooth projective real algebraic variety satisfies the real integral Hodge conjecture of \cite{benoistIntegralHodgeConjecture2020}. However, it is very useful as a framework to study problems in real algebraic geometry as it is intimately related to the size of the group $\Hr_\alg^c(X(\Rb),\Zb/2)$. Benoist and Wittenberg provide the first known examples of smooth projective varieties $X$ over $\Rb$ such that the restriction of the homomorphism $\pi:\Hr_G^{2c}(X(\Cb),\Zb(c))\rightarrow\Hr^c(X(\Rb),\Zb/2)$ to the subgroup $\Hdg_G^{2c}(X(\Cb),\Zb(c))_0\subseteq\Hr_G^{2c}(X(\Cb),\Zb(c))$ is surjective, yet the inclusion $\Hr_\alg^c(X(\Rb),\Zb/2)\varsubsetneq\Hr^c(X(\Rb),\Zb/2)$ of the subgroup of algebraic mod $2$ classes is strict. For such a variety $X$, the existence of non-algebraic classes in $\Hr^*(X(\Rb),\Zb/2)$ is not explained by Hodge theory or by topology: it must come from the failure of $X$ to satisfy the real integral Hodge conjecture.

In this approach, the remedy to the insufficiency of $\gamma_\BH^c(X)$ is to upgrade the \emph{target} to a more refined invariant of $X$, namely the equivariant cohomology of $(X(\Cb),\sigma)$: the \emph{source} remains the Chow groups of $X$. Motivic homotopy theory provides us with a set of algebraic tools that lead to a different proposition, where the Chow groups are lifted to Chow\emph{--Witt} groups: this is the main object of investigation of this paper.

\paragraph*{The quadratic real cycle class map.} Given a smooth variety $X$ over $\Rb$, this is a homomorphism \[\widetilde{\gamma}_\Rb^*(X):\widetilde{\CH}^*(X)\rightarrow\Hr^*(X(\Rb),\Zb)\] of graded rings: details can be found in \cite{hornbostelRealCycleClass2021} but we recall the construction below. Here the source $\widetilde{\CH}^*(X)$ is the \emph{Chow--Witt ring} of $X$ \cite{faselChowWittRing2007}, which refines the Chow ring of $X$ in the sense that there is a graded ring homomorphism $r:\widetilde{\CH}^*(X)\rightarrow\CH^*(X)$. The quadratic real cycle class map $\widetilde{\gamma}_\Rb^c(X)$ lifts the Borel--Haefliger cycle class map $\gamma_\BH^c(X)$ in the sense that the diagram
\begin{center}
\begin{tikzcd}
\widetilde{\CH}^c(X) \arrow[r,"r"] \arrow[d,swap,"\widetilde{\gamma}_\Rb^c(X)"] & \CH^c(X) \arrow[d,"\gamma_\BH^c(X)"] \\
\Hr^c(X(\Rb),\Zb) \arrow[r] & \Hr^c(X(\Rb),\Zb/2)
\end{tikzcd}
\end{center}
induced by the comparison map $r$ and reduction mod $2$ of the coefficients is commutative.

The most naive conjecture one can make about the quadratic real cycle class map $\widetilde{\gamma}_\Rb^c(X)$ is then to predict that this map is surjective. This is too optimistic: there are already obstructions in the case where $c=0$ and $\dim(X)=1$, namely the case of connected components of curves. One of the goals of this paper is to propose a refined statement that is more likely to be true. We eventually arrive at the following conjecture:

\begin{conj}
For every smooth $\Rb$-variety $X$ of dimension $d$ and every $0\leqslant c\leqslant d$, the image of $\widetilde{\gamma}_\Rb^c(X)$ contains $2^{d-c}\Hr^c(X(\Rb),\Zb)$. Moreover, the number $d-c$ in this exponent cannot be improved in general.
\end{conj}

We refer to Conjecture \ref{con:refined_hodge_conjecture} below for a more precise formulation.

\paragraph*{Contents.} The organisation of this paper is as follows. The first section is preliminary: we explain the objects of interest of the article and some of their basic properties, included a twisted version \[\widetilde{\gamma}_\Rb^c(X,\Lc):\widetilde{\CH}^c(X,\Lc)\rightarrow\Hr^c(X(\Rb),\Zb(L))\] of the constructions previously discussed; here the twist is a line bundle $\Lc$ on $X$. In the second section, we study the map $\widetilde{\gamma}_\Rb^d(X,\Lc)$ where $d$ is the dimension of $X$. The main observation here is that the analysis of $\widetilde{\gamma}_\Rb^d(X,\Lc)$ can be reduced for a number of questions to that of the mod $2$ cycle class maps studied in \cite{colliot-theleneZerocyclesCohomologyReal1996} thanks to Lemma \ref{lem:kernel_as_torsion} below. In the third section, we draw some immediate corollaries concerning the map $\widetilde{\gamma}_\Rb^{d-1}(X,\Lc)$ and we show conditional results on the map $\widetilde{\gamma}_\Rb^{d-2}(X,\Lc)$. The fourth section similarly contains results on the case of connected components, namely the map $\widetilde{\gamma}_\Rb^0(X,\Lc)$. In the final section, we first observe that studying the image of $\widetilde{\gamma}_\Rb^c(X,\Lc)$ is only relevant integrally: we prove a quantitative version of a result of Jacobson that implies that $\widetilde{\gamma}_\Rb^c(X,\Lc)$ is surjective for every $c$ and every $(X,\Lc)$ after inverting $2$ (see in particular the more precise Corollary \ref{cor:iso_after_inverting_2}). We then come to the aforementioned refined conjecture, taking twists into account, and we observe that the results of the previous section allow us to conclude that it holds true for curves and for surfaces, and in a number of significant cases for threefolds. Hornbostel independently studied similar questions to those investigated in this article in \cite{hornbostelFewComputationsReal2024}, especially quadratic real cycle class maps for surfaces, and there is some overlap with the present results (and proofs) which we signal below.

\paragraph*{Acknowledgements.} I wish to thank my advisor Jean Fasel for his guidance, his suggestions and his reading of this typescript. I also acknowledge the material support of Université Grenoble Alpes where I was a graduate student when the research that led to the present paper was conducted, the partial support of the ANR project ANR-21-CE40-0015 thanks to which I had multiple occasions to speak about the results proven below and the full support of the ANR project ANR-23-CE40-0011 during revision. Finally, further thanks are due to the reviewer for their detailed report on an earlier version of this article and their suggestions that led to significant improvements in the exposition of its results (leaving any remaining opacity solely on the author).

\section{Preliminaries}

\subsection{General considerations}

Given a set $S$, we denote by $\Zb[S]$ the free abelian group generated by $S$; it is naturally a ring if $S$ is a group.  We do not make any distinction between abelian groups $A$ with additive action of the group $G$ (namely, action by group automorphisms) and $\Zb[G]$-modules $A$. For example, if $G$ acts on a set $S$ on the right, then $\Zb[S]$ has a natural structure of right $\Zb[G]$-module inherited from this action. This extends to sheaves in an obvious way. 

\begin{defi}
Let $\Cs$ be a site and let $x$ be an object of $\Cs$. We say that $x$ is \emph{of cohomological dimension $\leqslant d$} if $\Hr^i(x,\Abf)=0$ for every $i>d$ and every abelian sheaf $\Abf$ on $\Cs$.
\end{defi}

\begin{exe}\label{exe:zariski_cohomological_dimension}
Let $X$ be a scheme. Then $X$ is \emph{of Zariski cohomological dimension $\leqslant d$} if $\Hr^i(X,\mathscr{F})=0$ for every $i>d$ and every sheaf $\mathscr{F}$ of abelian groups on the topological space $X$. If the underlying topological space of $X$ is Noetherian, then by \cite[Théorème 3.6.5]{grothendieckQuelquesPointsDalgebre1957}, the scheme $X$ is of Zariski cohomological dimension $\leqslant\dim(X)$ where $\dim(X)$ is the dimension of $X$ as a scheme.
\end{exe}

Let $(X,\Osc_X)$ be a locally ringed space. We use the terminologies of line bundles on $X$ and locally free $\Osc_X$-modules of rank $1$ interchangeably. These objects form a groupoid $\underline{\Pic}(X,\Osc_X)$ (we underline $\Pic$ to distinguish this category and its set of isomorphism classes of objects which, under the operation induced by the tensor product of modules, is the Picard group of $X$), denoted by $\underline{\Pic}(X)$ when no confusion on $\Osc_X$ can arise. We denote the residue field of $X$ at a point $x$ by $\kappa(x)$ and the maximal ideal of the local ring $\Osc_{X,x}$ by $\m_x$. Suppose now that $X$ is a scheme. The set of points $x$ of $X$ of codimension $p$ (that is, such that the irreducible closed subset $\overline{\{x\}}$ of $X$ is of codimension $p$) is denoted by $X^{(p)}$. Given a line bundle $\Lc$ on the scheme $X$ and a point $x\in X$, we set $\Lc(x)=\Lc\otimes\kappa(x)$. For every $x\in X$ such that $\m_x$ is finitely generated, we denote by $\Lambda_x$ the determinant of $(\m_x/\m_x^2)^*=\Hom_{\kappa(x)}(\m_x/\m_x^2,\kappa(x))$; thus $\Lambda_x$ is a $\kappa(x)$-vector space of dimension $1$. If $X$ is regular at $x$, then $\Lambda_x=\Lambda^p(\m_x/\m_x^2)^*$ where $p$ is the codimension of $x$.

Throughout we fix a field $k$ that is perfect, infinite of characteristic not $2$. Eventually we will take $k$ to be an extension of the field $\Rb$ of real numbers so the reader may safely assume that $k$ is of characteristic $0$ in this preliminary section.

\paragraph*{Twist by a line bundle.} Let $(\Cs,\Osc)$ be an essentially small ringed site, that is, an essentially small site $\Cs$ together with a sheaf $\Osc$ of rings. The reader may wish to think of $\Cs$ as a full subcategory of the category of $X$-schemes for some scheme $X$, endowed with a subcanonical topology, and of $\Osc$ as the sheaf $\Hom_X(\text{--},\Abb_X^1)$, or of $\Cs$ as the site of open subsets of a ringed space $(X,\Osc_X)$ and of $\Osc$ as the sheaf $\Osc_X$. An $\Osc$-module $\Lc$ is then locally free of rank $1$ if every $c\in\Cs$ has a covering sieve $R$ such that for every $c'\rightarrow c$ in $R$, there exists an isomorphism $\Lc_{|c'}\simeq\Osc_{|c'}$ of modules on $\Cs/c'$ with the induced ringed site structure.

Let $\Lc$ be a locally free $\Osc$-module of rank $1$. A section $s$ of $\Lc$ over $c\in\Cs$ is \emph{invertible} if the morphism $\Osc_{|c}\rightarrow\Lc_{|c}$ of modules over $\Cs/c$ carrying $a\in\Osc(c')$ to $as_{|c'}$ is an isomorphism. We denote by $\Lc^0\subseteq\Lc$ the subpresheaf of invertible sections. Since being invertible is a local property for a given section of $\Lc$, the presheaf $\Lc^0$ is a sheaf. The sheaf $\Osc^\times:u\mapsto\Osc(u)^\times$ of units of $\Osc$ acts on $\Lc^0$ by left multiplication and $\Lc^0$ is an $\Osc^\times$-torsor for this action, that is, the $\Osc^\times$-sheaf $\Lc^0$ is locally isomorphic to $\Osc^\times$ acting on itself by left multiplication. Now let $\Asc$ be a right $\Zb[\Osc^\times]$-module. We define $\Asc(\Lc)$ as the abelian sheaf on $\Cs$ associated with the presheaf $c\mapsto\Asc(c)\otimes_{\Zb[\Osc^\times(c)]}\Zb[\Lc^0(c)]$: we call $\Asc(\Lc)$ the twist of $\Asc$ by $\Lc$. For fixed $\Lc$, the construction $\Asc\mapsto\Asc(\Lc)$ is an exact functor from the category of $\Zb[\Osc^\times]$-modules to the category of abelian sheaves on $\Cs$. This is because exactness can be checked locally for sheaves and any local isomorphism $\Lc_{|c}\simeq\Osc_{|c}$ induces an isomorphism $\Asc_{|c}\simeq\Asc(\Lc)_{|c}$ for every right $\Zb[\Osc^\times]$-module $\Asc$. Moreover, since the construction $\Asc(\Lc)=\Asc\otimes_{\Zb[\Osc^\times]}\Zb[\Lc^0]$ is given by a tensor product, the functor $\Asc\mapsto\Asc(\Lc)$ commutes with colimits.

\begin{exe}\label{exe:twist_by_topological_line_bundle}
Let $X$ be a topological space. Denote by $\Csc_X$ the sheaf on $X$ assigning to each open subset $U$ of $X$ the ring of continuous maps from $U$ to $\Rb$. The pair $(X,\Csc_X)$ is then a locally ringed space. Let $\Csc_X^\times$ denote the sheaf of units of $\Csc_X$. Given a section $f$ of $\Csc_X^\times$, we denote by $\sign(f)$ the locally constant function to $\{\pm 1\}$ carrying $x$ to $+1$ if $f(x)>0$ and to $-1$ if $f(x)<0$: this is the \emph{sign} of $f$. The assignment $f\mapsto\sign(f)$ determines a morphism of sheaves of groups from $\Csc_X^\times$ to the sheaf of locally constant functions on $X$ to $\{\pm 1\}$, which we still denote by $\{\pm 1\}$. This last sheaf acts additively on every abelian sheaf on $X$ by multiplication so every abelian sheaf on $X$ can be regarded as a $\Zb[\Csc_X^\times]$-module. Therefore for every $L\in\underline{\Pic}(X,\Csc_X)$ and every abelian sheaf $\Asc$ on $X$, the twist $\Asc(L)$ of $\Asc$ by $L$ is defined. For example, if $2\Asc(U)=0$ for every open subset $U$ of $X$, then the action of $\{\pm 1\}$ on $\Asc$ is trivial, hence so is the action of $\Csc_X^\times$ and the sheaf $\Asc(L)$ is naturally isomorphic to $\Asc$. We will be particularly interested in the case where $\Asc$ is the constant sheaf $\Zb/2^N\Zb$ for some integer $N$. For every abelian group $A$ inducing a constant sheaf on $X$, which we still denote by $A$, the sheaf $A(L)$ is locally constant with stalk $A$ as the twist $\Asc(L)$ is isomorphic to $\Asc$ when $L$ is trivial.
\end{exe}

\paragraph*{The exponents of an abelian group.} We recall the definition of this notion:

\begin{defi}
Let $A$ be an abelian group and let $e\geqslant 0$ be an integer. We write $eA$ for the image of the endomorphism $x\mapsto ex$ of $A$. We say that $A$ \emph{has exponent $e$} if $eA=0$. More generally, a presheaf $\Fbf$ of abelian groups on a category $\Cs$ has exponent $e$ if $\Fbf(x)$ has exponent $e$ for every object $x$ of $\Cs$.
\end{defi}

Exponents will often be interpreted as bounding below the size of a subgroup in the following way.

\begin{exe}\label{exe:link_exponent_size}
If $A'$ is a subgroup of an abelian group $A$, then $A/A'$ has exponent $e$ if, and only if, we have an inclusion $eA\subseteq A'$ of subgroups of $A$.
\end{exe}

\begin{rema}\label{rema:exponents_cohomology_groups}
Although local epimorphisms are not sectionwise surjective, if $\Abf\rightarrow\Bbf$ is a local epimorphism of abelian presheaves on a site $\Cs$ where $\Abf$ has exponent $e$ and $\Bbf$ is a sheaf (for example, if $\Abf\rightarrow\Bbf$ is an epimorphism of sheaves or a sheafification morphism), then $\Bbf$ also has exponent $e$. This is because for any $\sigma\in\Bbf(c)$ where $c\in\Cs$, the equality $e\sigma=0$ \emph{locally} holds as $\sigma$ locally comes from a section of $\Abf$: since $\Bbf$ is a sheaf by assumption, this guarantees that $e\sigma=0$ in $\Bbf(c)$. For example, for every abelian sheaf $\Abf$ on $\Cs$, the sheaf $\Abf/e$ has exponent $e$ since there is a local epimorphism from the presheaf $c\mapsto\Abf(c)/e$ to $\Abf/e$. We also note that if $\Abf$ is an abelian sheaf on $\Cs$ having exponent $e$, then for every $x\in\Cs$ and every $i\geqslant 0$, the multiplication by $e$ endomorphism of $\Hr^i(x,\Abf)$ is induced by functoriality by the sheaf morphism $e\cdot\Id:\Abf\rightarrow\Abf$ and is therefore the zero map: in other words, the group $\Hr^i(x,\Abf)$ has exponent $e$.
\end{rema}

\subsection{$\Abb^1$-homotopy theory}

\paragraph*{Generalities.} We denote by $\Sm_k$ the full subcategory of the category $\Sch_k$ of all $k$-schemes spanned by separated smooth $k$-schemes of finite type. A $k$-scheme $X$ is \emph{essentially smooth} if it is the limit in $\Sch_k$ of a cofiltered diagram $(X_\alpha)_\alpha$ with affine étale transition maps such that $X_\alpha\in\Sm_k$ for every $\alpha$. The full subcategory of $\Sch_k$ spanned by essentially smooth $k$-schemes is denoted by $\Sm_k'$.

\begin{exe}\label{exe:localisation_essentially_smooth}
Let $X\in\Sm_k$ and let $S$ be a finite set of points of $X$ contained in an affine open of $X$. The \emph{semi-localisation} $X_S$ of $X$ at $S$ is the limit of the affine open neighbourhoods of $S$ in $X$ (along open immersions; such a morphism is always étale, and is an affine morphism if it is a morphism of affine schemes). If $S=\{x\}$ is a singleton, we simply call $X_{\{x\}}$ the localisation of $X$ at $x$. By definition, the $k$-scheme $X_S$ is essentially smooth. For example, if $X$ is connected with function field $K$, then $\Spec K$ is the localisation of $X$ at its generic point and is therefore essentially smooth over $k$. In fact, if $K/k$ is \emph{any} finitely generated field extension, then $K$ is the localisation of some connected $U\in\Sm_k$ at its generic point. To see this, note that there exists an integral finite-type $k$-algebra $A$ such that $K=\Frac A$. Since $k$ is perfect, there exists a non-empty open subset $U$ of $\Spec A$ which is smooth over $k$ and the field $K$ is then the localisation of $U$ at its generic point. It follows that the $k$-scheme $\Spec K$ is essentially smooth.
\end{exe}

Given a presheaf $\Fbf$ of sets on $\Sm_k$, we extend $\Fbf$ to essentially smooth $k$-schemes by setting $\Fbf(X)=\colim\Fbf(X_\alpha)$ if $X=\lim X_\alpha$ (this is well-defined by \cite[Proposition (8.13.5)]{grothendieckElementsGeometrieAlgebrique1966}). For instance, for every finitely generated field extension $K/k$, the set $\Fbf(K)$ is defined and usually denoted by $\Fr(K)$.

We endow $\Sm_k$ (resp. $\Sm_k'$) with the Nisnevich topology $\Nis$ of \cite{nisnevichCompletelyDecomposedTopology1989}. We denote by $\Hr_\Nis$ sheaf cohomology for this topology and by convention, unless otherwise specified, by sheaf on $\Sm_k$, we mean sheaf for the Nisnevich topology. The category of Nisnevich abelian sheaves on $\Sm_k$ is denoted by $\Abr_k$: every object of $\Abr_k$ extends to a Nisnevich sheaf on $\Sm_k'$ as described in the previous paragraph.

\begin{rema}\label{rema:colimit_extension_essentially_smooth}
Since filtered colimits commute with finite limits and with colimits, if $f:\Abf\rightarrow\Bbf$ is a morphism of sheaves of abelian groups on $\Sm_k$, then $(\Ker f)(X)=\colim\Ker f(X_\alpha)$ and $(\Coker f)(X)=\colim\Coker f(X_\alpha)$ for every essentially smooth $k$-scheme $X=\lim X_\alpha$.
\end{rema}

Let $X$ be a scheme. The small Nisnevich site $\Nis(X)$ of $X$ is the category of étale $X$-schemes endowed with the Nisnevich topology. Note that if $X$ underlies an object of $\Sm_k'$, then for every object $U\rightarrow X$ of $\Nis(X)$, the $k$-scheme $U$ is also an object of $\Sm_k'$. We denote by $\Zar(X)$ the site of open subschemes of the topological space $X$. By convention, by sheaf \emph{on $X$}, we mean sheaf on $\Zar(X)$, namely sheaf on the topological space $X$ in the usual sense.

\begin{rema}\label{rema:nisnevich_cohomological_dimension}
The Nisnevich cohomological dimension of $X\in\Sm_k$ is its cohomological dimension as an object of $(\Sm_k,\Nis)$: it is less than or equal to $\dim(X)$ by \cite[Theorem 1.32]{nisnevichCompletelyDecomposedTopology1989}.
\end{rema}


\paragraph*{Strictly $\Abb^1$-invariant sheaves, contractions and homotopy modules.}

\begin{defi}
A sheaf $\Mbf$ of abelian groups on $\Sm_k$ is \emph{strictly $\Abb^1$-invariant} if for every $X\in\Sm_k$, the projection map $p:X\times\Abb^1\rightarrow X$ induces an isomorphism $p^*:\Hr_\Nis^i(X,\Mbf)\rightarrow\Hr_\Nis^i(X\times\Abb^1,\Mbf)$ of abelian groups for every $i\geqslant 0$. Strictly $\Abb^1$-invariant abelian sheaves form a full subcategory $\Abr_k^{\Abb^1}$ of the category $\Abr_k$ of abelian sheaves on $\Sm_k$.
\end{defi}

\begin{theo}[Morel]\label{theo:strictly_invariant_sheaves_abelian}
The category $\Abr_k^{\Abb^1}$ of strictly $\Abb^1$-invariant sheaves is abelian and the fully faithful embedding of $\Abr_k^{\Abb^1}$ into the category $\Abr_k$ of abelian sheaves on $\Sm_k$ is exact.
\end{theo}

\begin{proof}
See \cite[Corollary 6.24]{morelA1AlgebraicTopologyField2012}.
\end{proof}

In particular, the kernel and cokernel in $\Abr_k$ of a morphism of strictly $\Abb^1$-invariant sheaves are again strictly $\Abb^1$-invariant sheaves.

\begin{defi}
Let $\Mbf$ be a sheaf of abelian groups on $\Sm_k$. The \emph{contraction} of $\Mbf$ is the sheaf $\Mbf_{-1}$ given by $\Mbf_{-1}(X)=\Coker(p^*:\Mbf(X)\rightarrow\Mbf(X\times\mathbb{G}_{m,k}))$ where $p:X\times\mathbb{G}_{m,k}\rightarrow X$ is the projection. We can iterate this construction and define $\Mbf_{-n}$ as the sheaf obtained from $\Mbf$ by applying the contraction $n$ times for every $n\geqslant 0$ (hence $\Mbf_{-0}=\Mbf$). These constructions extend to endofunctors $\Mbf\mapsto\Mbf_{-n}$ of $\mathsf{Ab}_k$.
\end{defi}

\begin{prop}
If $\Mbf$ is a strictly $\Abb^1$-invariant sheaf on $\Sm_k$, then so is its contraction $\Mbf_{-1}$.
\end{prop}

\begin{proof}
Combine \cite[Lemma 2.32, Corollary 5.45]{morelA1AlgebraicTopologyField2012}.
\end{proof}

\begin{defi}
A \emph{homotopy module} is a $\Zb$-graded strictly $\Abb^1$-invariant sheaf $(\Mbf_n)_{n\in\Zb}$ together with a family $\omega_n:\Mbf_{n-1}\rightarrow(\Mbf_n)_{-1}$ of sheaf isomorphisms. A morphism of homotopy modules from $(\Mbf,\omega)$ to $(\Mbf',\omega')$ is a morphism $f:\Mbf\rightarrow\Mbf'$ of graded sheaves (homogeneous of degree $0$) such that $\omega'_n\circ f_{n-1}=(f_n)_{-1}\circ\omega_n$ for every $n$. We denote the category of homotopy modules by $\mathsf{HM}(k)$.
\end{defi}

\begin{exe}
Given a homotopy module $(\Mbf,\omega)$, for every $m\in\Zb$, we denote by $\Mbf_{*+m}$ the homotopy module given by $(\Mbf_{*+m})_n=\Mbf_{n+m}$ and the isomorphisms $(\Mbf_{*+m})_{n-1}=\Mbf_{n+m-1}\xrightarrow{\omega}(\Mbf_{n+m})_{-1}=((\Mbf_{*+m})_n)_{-1}$ of sheaves.
\end{exe}

\begin{prop}\label{prop:category_of_homotopy_modules}
The category of homotopy modules is abelian and its embedding in the category of graded abelian sheaves is exact.
\end{prop}


\begin{proof}
Denote by $\mathcal{SH}(k)$ the $\mathbb{P}^1$-stable $\Abb^1$-homotopy category of \cite{morelIntroductionA1homotopyTheory2003}. This is a triangulated category, and the Eilenberg--Mac Lane spectrum construction induces an equivalence from the category $\mathsf{HM}(k)$ to the heart of $\mathcal{SH}(k)$, see \cite[Theorem 5.2.6]{morelIntroductionA1homotopyTheory2003}, which is then abelian as is the heart of every triangulated category.
\end{proof}

\paragraph*{Abelian sheaves with $\Gm$-action.} We endow $\Sm_k$ with the sheaf of rings $\Osc_{k}:U\mapsto\Osc_U(U)$ represented by $\Abb_k^1$. Let $\Mbf$ be an abelian sheaf on $\Sm_k$ with an additive action of $\Osc_{k}^\times$ on the right. Let $X$ be an essentially smooth $k$-scheme and let $\Lc$ be a line bundle on $X$. Then $\Lc$ induces a locally free module of rank $1$ over the ringed site $(\Nis(X),\Osc_{\Nis(X)}:U\mapsto\Osc_U(U))$, whose sections over $f:U\rightarrow X$ are the global sections of the $\Osc_U$-module $f^*\Lc$; we still denote this $\Osc_{\Nis(X)}$-module by $\Lc$. We write $\Mbf(\Lc)$ for the twist of the restricted sheaf $\Mbf_{|\Nis(X)}$ with $\Osc_{\Nis(X)}^\times$-action by the locally free $\Osc_{\Nis(X)}$-module $\Lc$ of rank $1$. One can also consider the restriction $\Mbf_{|\Zar(X)}$ of $\Mbf$ to the ringed site $(\Zar(X),\Osc_X)$ and twist this restriction by the line bundle $\Lc$. It then follows from the proof of \cite[Lemma 2.24]{hornbostelRealCycleClass2021} that the resulting sheaf on $X$ is the restriction of $\Mbf(\Lc)$ to $\Zar(X)$ so both sheaves will be denoted by $\Mbf(\Lc)$. For instance, given a finitely generated field extension $F/k$ and an $F$-vector space $L$ of dimension $1$ inducing a line bundle $\Lc$ on $\Spec F$, the twisted group $\Hr^0(\Spec F,\Mbf(\Lc))$ of sections is defined and is usually denoted by $\Mr(F,L)$. For fixed $\Lc\in\underline{\Pic}(X)$, the construction $\Mbf\mapsto\Mbf(\Lc)$ is functorial with respect to morphisms of $\Zb[\Osc_{k}^\times]$-modules on $\Sm_k$ and induces an exact functor from the category of $\Zb[\Osc_k^\times]$\hbox{-}modules on $\Sm_k$ to the category of abelian sheaves on $\Nis(X)$ (respectively on $\Zar(X)$) which commutes with colimits.

\begin{exe}\label{exe:gm_action_on_contraction}
Let $\Mbf$ be an abelian sheaf on $\Sm_k$. Then the contraction $\Mbf_{-1}$ has a natural additive action of $\Osc_k^\times$, described \emph{e.g.} at \cite[p. 78]{morelA1AlgebraicTopologyField2012}, and can thus be twisted by a line bundle.
\end{exe}

\paragraph*{Rost--Schmid complexes.} Let $\Mbf$ be a strictly $\Abb^1$-invariant sheaf and let $U\in\Sm_k'$. Then the \emph{Rost--Schmid complex} $\Cr_\RS(U;\Mbf)$ is a complex of abelian groups constructed in \cite[Chapter 5]{morelA1AlgebraicTopologyField2012}. Its components are of the form \[\Cr_\RS(U;\Mbf)^p=\bigoplus_{x\in U^{(p)}}\Mr_{-p}(\kappa(x),\Lambda_x)\] where we recall that $\Lambda_x$ is the determinant of $(\m_x/\m_x^2)^*$, and its differentials are defined using residue homomorphisms and transfers. This construction is contravariant in étale morphisms. In particular, for every essentially smooth $k$-scheme $X$, this yields presheaves $U\mapsto\Cr_\RS(U;\Mbf)^p$ on $\Nis(X)$ and on $\Zar(X)$ that are in fact sheaves for the Nisnevich and Zariski topology respectively. Localisation at the codimension $0$ points yields a map $\Mbf(U)\rightarrow\bigoplus_{x\in U^{(0)}}\Mr(\kappa(x))$ for every $U\in\Nis(X)$ (the $\kappa(x)$-vector space $\Lambda_x$ is \emph{canonically} isomorphic to $\kappa(x)$ if $x$ has codimension $0$). The resulting complex $0\rightarrow\Mbf_{|Nis(X)}\rightarrow\Cr_\RS(\text{--};\Mbf)^0\rightarrow\cdots\rightarrow\Cr_\RS(\text{--};\Mbf)^d\rightarrow 0$ is then a resolution of $\Mbf_{|\Nis(X)}$, and restricts to a resolution of $\Mbf_{|\Zar(X)}$, by flasque sheaves, see \cite[Theorem 5.41, Lemma 5.42]{morelA1AlgebraicTopologyField2012}. Consequently, for every $i$, the Nisnevich and Zariski cohomology groups of $X$ with coefficients in $\Mbf$ in degree $i$ are both given by the cohomology group $\Hr^i(\Cr_\RS(X;\Mbf))$ of the complex $\Cr_\RS(X;\Mbf)$ (\cite[Corollary 5.43]{morelA1AlgebraicTopologyField2012}). We denote these coinciding Nisnevich and Zariski cohomology groups by $\Hr^i(X,\Mbf)$.

Now assume that $\Mbf\cong\Nbf_{-1}$ is the contraction of a strictly $\Abb^1$-invariant sheaf $\Nbf$. Then as noted in Example \ref{exe:gm_action_on_contraction}, the sheaf $\Mbf$ has a natural $\Zb[\Osc_{k}^\times]$-module structure: consequently, every component and differential of the Rost--Schmid complex of $\Mbf$ on an essentially smooth $k$-scheme $X$, including the cohomological degree $0$, can be twisted by any line bundle $\Lc$ on $X$, see \cite[Remark 5.13, 5.14]{morelA1AlgebraicTopologyField2012}. This yields a complex $\Cr_\RS(Y,f^*\Lc;\Mbf)$ for every object $f:Y\rightarrow X$ of $\Nis(X)$, usually more simply denoted by $\Cr_\RS(Y,\Lc;\Mbf)$, and thus a complex $0\rightarrow\Mbf(\Lc)\rightarrow\Cr_\RS(\text{--},\Lc;\Mbf)^0\rightarrow\cdots\rightarrow\Cr_\RS(\text{--},\Lc;\Mbf)^d\rightarrow 0$ of sheaves on $\Nis(X)$ and on $\Zar(X)$. As in the untwisted case, this complex is in fact a flasque resolution of the restriction of $\Mbf(\Lc)$ to $\Nis(X)$ (resp. to $\Zar(X)$). Consequently, as before, the Nisnevich and Zariski cohomology groups of $X$ with coefficients in $\Mbf(\Lc)$ coincide and are denoted by $\Hr^*(X,\Mbf(\Lc))$, and they can be computed using the Rost--Schmid complex: the group $\Hr^i(X,\Mbf(\Lc))$ coincides with $\Hr^i(\Cr_\RS(X,\Lc;\Mbf))$ for every $i$. This twisted case follows from Morel's results.

The Rost--Schmid complex allows us to deduce many properties of strictly $\Abb^1$-invariant sheaves from corresponding properties of their values on \emph{fields}. Here is a very easy incarnation of this principle.

\begin{exe}
Let $\Mbf$ be an abelian sheaf on $\Sm_k$. If $\Mr(F)=0$ for every finitely generated field extension $F/k$, then $\Mbf=0$ since $\Mbf(U)$ embeds into $\Cr_\RS(U;\Mbf)^0=\bigoplus_{x\in U^{(0)}}\Mr(\kappa(x))$ for every $U\in\Sm_k$ (see also \cite[Lemma 3.45]{morelA1AlgebraicTopologyField2012}; it would suffice to know that strictly $\Abb^1$-invariant sheaves are unramified in the sense of \cite[Definition 2.1]{morelA1AlgebraicTopologyField2012}). Thus by Remark \ref{rema:colimit_extension_essentially_smooth}, a morphism $f:\Mbf\rightarrow\Mbf'$ of strictly $\Abb^1$-invariant sheaves is an isomorphism if, and only if, for every finitely generated field extension $F/k$, the morphism $f(F):\Mr(F)\rightarrow\Mr'(F)$ induced by $f$ is an isomorphism of abelian groups.

Note that if $\Mbf$ is a contraction, then we have a non-canonical isomorphism $\Mr(F,L)\simeq\Mr(F)$ of groups for every $F$-vector space $L$ of dimension $1$. Consequently, if $\Mr(F)=0$ for every finitely generated field extension $F/k$, then $\Mbf(\Lc)$ is the zero sheaf on $\Nis(X)$ for every line bundle $\Lc$ on an essentially smooth $k$-scheme $X$.
\end{exe}

\subsection{Milnor--Witt $\Kr$-theory}

Recall from \cite[Definition 3.1]{morelA1AlgebraicTopologyField2012} the definition of the Milnor--Witt $\Kr$-theory $\Kr_*^\MW(F)$ of a field $F$. It is a $\Zb$-graded ring defined by generators $[a]$ in degree $+1$ for every $a\in F^\times$ and $\eta$ in degree $-1$, and relations. The map $F^\times\rightarrow\Kr_0^\MW(F)$ carrying $a$ to $1+\eta[a]$ is multiplicative so one defines an additive action of $F^\times$ on $\Kr_0^\MW(F)$ on the right by setting $\gamma\cdot a=\gamma(1+\eta[a])$, the product being taken in the ring $\Kr_0^\MW(F)$. Using the product $\Kr_n^\MW(F)\times\Kr_0^\MW(F)\rightarrow\Kr_n^\MW(F)$ underlying the graded ring $\Kr_*^\MW(F)$, we obtain a $\Zb[F^\times]$-module structure on $\Kr_n^\MW(F)$ for every $n\geqslant 0$ and thus a twisted group $\Kr_n^\MW(F,L)$ for every $F$-vector space $L$ of dimension $1$. This twisted group sits in a fibre product diagram (\cite[Theorem 5.3]{morelPuissancesLidealFondamental2004}) with exact row:
\begin{equation}\label{eq:mw_fibre_product_field}
\begin{tikzcd}
            &                          & \Kr_n^\MW(F,L) \arrow[r] \arrow[d,two heads] \arrow[dr, phantom, "\usebox\pullback" , very near start, color=black] & \Kr_n^\Mr(F) \arrow[d,two heads] & \\
0 \arrow[r] & \Ir^{n+1}(F,L) \arrow[r] & \Ir^n(F,L) \arrow[r] & \overline{\Ir}^n(F) \arrow[r] & 0
\end{tikzcd}
\end{equation}

Let us explain the objects that appear in this square. The group $\Kr_n^\Mr(F)$ is the $n$-th Milnor $\Kr$-theory group of $F$ introduced in \cite{milnorAlgebraicKtheoryQuadratic1970} and thus vanishes if $n<0$; the group $F^\times$ acts trivially on $\Kr_n^\Mr(F)$. If $n\geqslant 1$, we let $\Ir^n(F)$ denote the $n$-th power of the ideal $\Ir(F)$ of even rank forms in the Witt ring $\Wr(F)$ of symmetric bilinear forms on $F$ and we set $\Ir^n(F)=\Wr(F)$ if $n\leqslant 0$. The group $F^\times$ acts on $\Wr(F)$ by multiplication by the Witt class of the rank one forms $\langle a\rangle:(x,y)\mapsto axy$ for $a\in F^\times$. By its very definition, this action preserves the rank, thus it induces an action on $\Ir(F)$, hence on $\Ir^n(F)$ for every $n$. Consequently, the twisted groups $\Ir^n(F,L)$ are defined. The inclusion $\Ir^{n+1}(F)\rightarrow\Ir^n(F)$ is $F^\times$-equivariant hence there is an induced action on $\overline{\Ir}^n(F)=\Ir^n(F)/\Ir^{n+1}(F)$. This action is trivial by \cite[Lemme E.1.3]{faselGroupesChowWitt2008} thus there is a natural isomorphism $\overline{\Ir}^n(F,L)\cong\overline{\Ir}^n(F)$. The bottom right horizontal map of Diagram (\ref{eq:mw_fibre_product_field}) is then the quotient map $\Ir^n(F,L)\rightarrow\overline{\Ir}^n(F)$. The right vertical map is onto, given on symbols in Milnor $\Kr$-theory by $\{a_1,\ldots,a_n\}\mapsto\llangle a_1\rrangle\otimes\cdots\otimes\llangle a_n\rrangle$ where $\llangle a\rrangle=\langle 1,-a\rangle$.

The abelian groups that appear in the above fibre product square underly data that satisfy the hypotheses of \cite[Theorem 2.46]{morelA1AlgebraicTopologyField2012}. For Milnor--Witt $\Kr$-theory, this is checked in \cite[Section 3.2]{morelA1AlgebraicTopologyField2012}. It follows for the other data by \cite[Lemma 3.32]{morelA1AlgebraicTopologyField2012} since, by \cite[Examples 3.33, 3.34]{morelA1AlgebraicTopologyField2012}, they are \emph{unramified $\Kr^{\mathcal{R}}$-theories} in the sense of \cite[p. 69]{morelA1AlgebraicTopologyField2012}. In particular, we have residue homomorphisms for Milnor--Witt $\Kr$-theory (depending on a choice of uniformising parameter) that, in view of the defining formulas of \cite[Theorem 3.15]{morelA1AlgebraicTopologyField2012}, \cite[Lemma 2.1]{milnorAlgebraicKtheoryQuadratic1970} and \cite[Chapter IV, Lemmas (1.2) and (1.4)]{milnorSymmetricBilinearForms1973}, induce the usual residue for Milnor $\Kr$-theory and for the powers of the fundamental ideal of the Witt ring. According to \cite[Theorem 2.46]{morelA1AlgebraicTopologyField2012}, there are induced unramified sheaves $\Kbf_n^\MW$, $\Kbf_n^\Mr$, $\Ibf^n$ and $\overline{\Ibf}^n$ which are strictly $\Abb^1$-invariant as a consequence of \cite[Theorem 2.46, Corollary 5.45]{morelA1AlgebraicTopologyField2012}. For every $n\leqslant 0$, we write $\Ibf^n=\Wbf$, and we set $\overline{\Ibf}^0=\overline{\Wbf}$. The sheaves $\Kbf_n^\MW$, $\Kbf_n^\Mr$, $\Ibf^n$ and $\overline{\Ibf}^n$ underly $\Zb[\Osc_{k}^\times]$-modules (the action of $\Osc_k^\times$ on $\Kbf_n^\Mr$ and on $\overline{\Ibf}^n$ being trivial) and can therefore be twisted by a line bundle. We have the following computation of contractions.

\begin{prop}
For every $n\in\Zb$, there is an isomorphism $(\Kbf_n^\MW)_{-1}\cong\Kbf_{n-1}^\MW$ of sheaves on $\Sm_k$ which descends to isomorphisms of sheaves of the form $(\Kbf_n^\Mr)_{-1}\cong\Kbf_{n-1}^\Mr$, $(\Ibf^n)_{-1}\cong\Ibf^{n-1}$ and $(\overline{\Ibf}^n)_{-1}\cong\overline{\Ibf}^{n-1}$.
\end{prop}

\begin{proof}
This is the isomorphism of \cite[Lemma 2.48]{morelA1AlgebraicTopologyField2012}.
\end{proof}

Consequently we get homotopy modules $\Kbf_*^\MW$, $\Kbf_*^\Mr$, $\Ibf^*$ and $\overline{\Ibf}^*$.

\begin{rema}
The $\Osc_{k}^\times$-actions that these sheaves acquire as contractions as explained in Example \ref{exe:gm_action_on_contraction} are the ones induced by the $F^\times$-module structure on $\Kr_0^\MW(F)$ \cite[Lemma 3.49]{morelA1AlgebraicTopologyField2012}.
\end{rema}

\paragraph*{Chow--Witt groups.} Using the $\Zb[\Osc_{k}^\times]$-module structures described above, one constructs twisted sheaves $\Kbf_n^\MW(\Lc)$ and $\Ibf^n(\Lc)$ for every locally free module $\Lc$ of rank $1$ on an essentially smooth $k$-scheme $X$ (we recall that $\Osc_{k}^\times$ acts trivially on $\Kbf_*^\Mr$ and $\overline{\Ibf}^*$). There is a fibre product square:
\begin{equation}\label{eq:mw_fibre_product}
\begin{tikzcd}
\Kbf_n^\MW(\Lc) \arrow[r] \arrow[d] & \Kbf_n^\Mr \arrow[d] \\
\Ibf^n(\Lc) \arrow[r] & \overline{\Ibf}^n
\end{tikzcd}
\end{equation}
of sheaves on $X$ in this situation (and a similar fibre product square of sheaves on $\Sm_k$ in the untwisted case). We can now define Chow--Witt groups:

\begin{defi}
The \emph{$c$-th Chow--Witt group} of an essentially smooth $k$-scheme $X$ twisted by a line bundle $\Lc$ is $\widetilde{\CH}^c(X,\Lc)=\Hr^c(X,\Kbf_c^\MW(\Lc))$.
\end{defi}

Chow--Witt groups refine Chow groups in the following sense. The Rost--Schmid complex $\Cr_\RS(X,\Kbf_c^\Mr)$ exhibits $\Hr^c(X,\Kbf_c^\Mr)$ as the cokernel \[\bigoplus_{x\in X^{(c-1)}}\kappa(x)^\times\xrightarrow{\partial_\RS}\bigoplus_{x\in X^{(c)}}\Zb\rightarrow\Hr^c(X,\Kbf_c^\Mr)\rightarrow 0\] since for every field $F$, we have isomorphisms $\Kr_1^\Mr(F)\cong F^\times$ and $\Kr_0^\Mr(F)\cong\Zb$ and the group $\Kr_{-1}^\Mr(F)$ vanishes. It turns out that the relevant differential $\partial_\RS$ of the Rost--Schmid complex is the divisor map of \cite[Section 1.3]{fultonIntersectionTheory1998} hence $\Hr^c(X,\Kbf_c^\Mr)=\CH^c(X)$. Thus there is a comparison map $r:\widetilde{\CH}^c(X,\Lc)\rightarrow\CH^c(X)$ induced by the morphism $\Kbf_c^\MW(\Lc)\rightarrow\Kbf_c^\Mr$ of sheaves. Tensoring the above exact sequence with $\Zb/2$ also yields an exact sequence \[\bigoplus_{x\in X^{(c-1)}}\Kr_1^\Mr(\kappa(x))/2\xrightarrow{\partial_\RS}\bigoplus_{x\in X^{(c)}}\Kr_0^\Mr(\kappa(x))/2\rightarrow\CH^c(X)/2\rightarrow 0\] since $(\Kbf_n^\Mr/2)_{-m}\cong\Kbf_{n-m}^\Mr/2$ by Proposition \ref{prop:category_of_homotopy_modules}, so that $\CH^c(X)/2=\Hr^c(X,\Kbf_c^\Mr/2)$. Moreover, the Nisnevich cohomological dimension of $X$ is bounded above by its Krull dimension as noted in Remark \ref{rema:nisnevich_cohomological_dimension} and thus $\widetilde{\CH}^n(X,\Lc)$ vanishes if $n>\dim(X)$, like the usual Chow groups.

\begin{lem}\label{lem:vanishing_cohomology_k_milnor}
Let $m\geqslant n$ be integers and let $X\in\Sm_k$. Then the group $\Hr^{m+1}(X,2\Kbf_n^\Mr)$ vanishes.
\end{lem}

\begin{proof}
By Proposition \ref{prop:category_of_homotopy_modules}, the isomorphism $(\Kbf_n^\Mr)_{-p}\cong\Kbf_{n-p}^\Mr$ of sheaves induces an isomorphism $(2\Kbf_n^\Mr)_{-p}\cong 2\Kbf_{n-p}^\Mr$ for every $p\geqslant 0$. Thus $(2\Kbf_n^\Mr)_{-(m+1)}\cong 2\Kbf_{n-m-1}^\Mr$: since Milnor $\Kr$-theory vanishes in negative degrees and as $n-m-1<0$, it follows that $(2\Kr_n^\Mr)_{-(m+1)}(F)\cong 2\Kr_{n-m-1}^\Mr(F)=0$ for every finitely generated extension $F$ of $k$. Now the Rost--Schmid complex of $2\Kbf_n^\Mr$ exhibits $\Hr^{m+1}(X,2\Kbf_n^\Mr)$ as a subquotient of $\bigoplus_{x\in X^{(m+1)}}(2\Kr_n^\Mr)_{-(m+1)}(\kappa(x))$ which vanishes. \emph{A fortiori}, the group $\Hr^{m+1}(X,2\Kbf_n^\Mr)$ vanishes as required.
\end{proof}

The following lemma will allow us to substitute the cohomology of $\Ibf^n(\Lc)$ for the Chow--Witt groups in many questions of images for quadratic real cycle class maps.

\begin{lem}\label{lem:milnor_k_theory_to_i_cohomology_surjective}
Let $X\in\Sm_k$, let $\Lc$ be a line bundle on $X$ and let $m\geqslant n$ be integers. Then the map $\Kbf_n^\MW\rightarrow\Ibf^n(\Lc)$ of sheaves of Diagram \emph{(\ref{eq:mw_fibre_product})} induces an epimorphism $\Hr^m(X,\Kbf_n^\MW(\Lc))\rightarrow\Hr^m(X,\Ibf^n(\Lc))$ of abelian groups.
\end{lem}

This lemma applies in particular to the map $\Hr^n(X,\Kbf_n^\MW(\Lc))=\widetilde{\CH}^n(X,\Lc)\rightarrow\Hr^n(X,\Ibf^n(\Lc))$.

\begin{proof}
By the affirmation of the Milnor conjecture on quadratic forms in \cite{orlovExactSequenceKM22007}, the morphism $\Kbf_n^\Mr\rightarrow\overline{\Ibf}^n$ of sheaves given on symbols by $\{a_1,\ldots,a_n\}\mapsto\llangle a_1\rrangle\otimes\cdots\otimes\llangle a_n\rrangle$ factors into an isomorphism $\Kbf_n^\Mr/2\cong\overline{\Ibf}^n$ so that we have an exact sequence \[0\rightarrow 2\Kbf_n^\Mr\rightarrow\Kbf_n^\Mr\rightarrow\overline{\Ibf}^n\rightarrow 0\] of abelian sheaves on $\Sm_k$. By the fibre product square (\ref{eq:mw_fibre_product}) of sheaves on $X$ (and commutation of pullbacks with kernels which are particular cases of limits), we get an exact sequence \[0\rightarrow 2\Kbf_n^\Mr\rightarrow\Kbf_n^\MW(\Lc)\rightarrow\Ibf^n(\Lc)\rightarrow 0\] of sheaves on $X$. It induces an exact sequence \[\Hr^m(X,\Kbf_n^\MW(\Lc))\rightarrow\Hr^m(X,\Ibf^n(\Lc))\rightarrow\Hr^{m+1}(X,{2\Kbf}_n^\Mr)\] of abelian groups. The last group in this exact sequence vanishes by Lemma \ref{lem:vanishing_cohomology_k_milnor} hence the homomorphism $\Hr^m(X,\Kbf_n^\MW(\Lc))\rightarrow\Hr^m(X,\Ibf^n(\Lc))$ is surjective as required.
\end{proof}

\subsection{Comparing sheafification and unramified sheaves}

In this subsection, we compare sheafification of global constructions and unramified sheaves in two contexts: mod $2$ étale cohomology and Witt groups. The results of this subsection (which are well-known to experts) guarantee that we can use constructions and results that rely on either seamlessly.

\paragraph*{Unramified étale cohomology.} Let $F$ be a field of characteristic different from $2$. For every $n\geqslant 0$, we denote by $\Hr_\et^n(F,\Zb/2)$ the $n$-th étale cohomology group of $F$ with coefficients in $\Zb/2$. The graded abelian group $\Hr_\et^*(F,\Zb/2)$ underlies a graded $\Zb/2$-algebra whose multiplication is the cup-product $\cup$. Given $a\in F^\times$, we denote by $(a)\in\Hr_\et^1(F,\Zb/2)$ the associated symbol in Galois cohomology. The assignment $\{a_1,\ldots,a_n\}\mapsto(a_1)\cup\cdots\cup(a_n)$ extends to a well-defined homomorphism $\Kr_n^\Mr(F)/2\rightarrow\Hr_\et^n(F,\Zb/2)$ of groups and these assemble into a morphism $\Kr_*^\Mr(F)/2\rightarrow\Hr_\et^*(F,\Zb/2)$ of graded rings. By the affirmation of the Milnor conjecture in \cite{voevodskyMotivicCohomology2coefficients2003}, the morphism $\Kr_*^\Mr(F)/2\rightarrow\Hr_\et^*(F,\Zb/2)$ is an isomorphism of graded rings. Moreover, as was already used in the proof of Lemma \ref{lem:milnor_k_theory_to_i_cohomology_surjective}, the morphism $\Kr_*^\Mr(F)/2\rightarrow\overline{\Ir}^*(F)$ given on symbols by $\{a_1,\ldots,a_n\}\mapsto\llangle a_1\rrangle\otimes\cdots\otimes\llangle a_n\rrangle$ is an isomorphism of graded rings by \cite{orlovExactSequenceKM22007}. In particular, we have an isomorphism $\overline{\Ir}^*(F)\cong\Hr_\et^*(F,\Zb/2)$ modulo which for every unit $a$ in $F^\times$, the element $\llangle a\rrangle=\langle 1,-a\rangle$ of $\overline{\Ir}^1(F)$ corresponds to $(a)\in\Hr_\et^1(F,\Zb/2)$.

The groups $\Hr_\et^n(\text{--},\Zb/2)$ underly an unramified $\Fs_k$-datum in the sense of \cite[Definition 2.9]{morelA1AlgebraicTopologyField2012} and thus induce a sheaf $\Hbf^n$ on $\Sm_k$. These sheaves arrange into a homotopy module $\Hbf^*$ by the results of \cite[Section 2.3]{morelA1AlgebraicTopologyField2012}. By \emph{e.g.} \cite[Propositions 101.8, 101.9]{elmanAlgebraicGeometricTheory2008}, the isomorphisms $\overline{\Ir}^*(F)\cong\Kr_*^\Mr(F)/2\cong\Hr_\et^*(F,\Zb/2)$ provided by the affirmation of the Milnor conjecture are compatible with residue homomorphisms: consequently, they determine isomorphisms $\overline{\Ibf}^*\cong\Kbf_*^\Mr/2\cong\Hbf^*$ of homotopy modules. On the other hand, given a smooth $k$-variety $X$, one can consider the sheaf $\mathscr{H}^n$ on $X$ associated with the presheaf $U\mapsto\Hr_\et^n(U,\Zb/2)$. For every open subscheme $U$ of $X$, there is a localisation morphism $\Hr_\et^n(U,\Zb/2)\rightarrow\bigoplus_{x\in U^{(0)}}\Hr_\et^n(\kappa(x),\Zb/2)$ that factors through the subgroup $\Hbf^n(U)$. This construction determines a morphism $\mathscr{H}^n\rightarrow\Hbf^n$ of sheaves on $X$.

\begin{theo}\label{theo:unramified_étale_cohomology}
Let $X$ be a smooth $k$-variety. Then the map $\mathscr{H}^n\rightarrow\Hbf^n$ of sheaves on $X$ is an isomorphism.
\end{theo}

\begin{proof}
Let $Y$ be an essentially smooth $k$-scheme. In \cite{blochGerstensConjectureHomology1974}, the authors construct a cohomological coniveau spectral sequence with $\Er_1$-page of the form $\Er_1^{p,q}(Y)=\bigoplus_{x\in Y^{(p)}}\Hr_\et^{q-p}(\kappa(x),\Zb/2)$, converging to $\Hr_\et^{p+q}(Y,\Zb/2)$. By their main result \cite[(4.2) Theorem]{blochGerstensConjectureHomology1974} in the local case, extended to semi-local schemes in \cite[Proposition 2.1.2]{colliot-theleneBlochOgusGabberTheorem1997}, the line $q=n$ of the $\Er_1$-page is an acyclic complex if $Y$ is the semi-localisation of a smooth $k$-variety at a finite set. Therefore the sheaves $U\mapsto\Er_1^{p,n}(U)$ on $X$, together with the differentials $d_1^{p,n}$ of the $\Er_1$-page, induce a resolution of $\Hsc^n$ by acyclic sheaves. This computes the $\Er_2$-page as $\Er_2^{p,q}(X)=\Hr^p(X,\Hsc^q)$ and, taking $p=0$, shows that the sequence of abelian groups \[0\rightarrow\Hsc^n(U)\rightarrow\bigoplus_{x\in U^{(0)}}\Hr_\et^n(\kappa(x),\Zb/2)\xrightarrow{d_1^{0,n}}\bigoplus_{x\in U^{(1)}}\Hr_\et^{n-1}(\kappa(x),\Zb/2)\] is exact for every open subset $U$ of $X$. By \cite[Proposition 1.3]{colliot-theleneVarietesUnirationnellesNon1989}, the differential $d_1^{0,n}$ is given by the usual residue homomorphism $\partial$ in Galois cohomology defined in \emph{e.g.} \cite[Subsection 99.5]{elmanAlgebraicGeometricTheory2008}. But by the very definition of the homotopy module $\Hbf^*$, the group $\Hbf^n(U)$ consists of the unramified classes in $\bigoplus_{x\in U^{(0)}}\Hr_\et^n(\kappa(x),\Zb/2)$, namely it is the kernel of the morphism $\bigoplus_{x\in U^{(0)}}\Hr_\et^n(\kappa(x),\Zb/2)\xrightarrow{\partial}\bigoplus_{x\in U^{(1)}}\Hr_\et^{n-1}(\kappa(x),\Zb/2)$. Hence the map $\Hsc^n(U)\rightarrow\Hbf^n(U)$ is an isomorphism, as required.
\end{proof}

Thus we do not have to distinguish between $\mathscr{H}^n$ and $\Hbf^n$ on smooth $k$-varieties. In the sequel, we will use the notation $\mathscr{H}^n$ (as in \cite{colliot-theleneZerocyclesCohomologyReal1996}).

\paragraph*{Symmetric bilinear forms on schemes.} Let $X$ be a scheme and let $\Lc\in\underline{\Pic}(X)$. These data determine an exact category with duality structure on the category of locally free $\Osc_X$-modules of finite rank in the sense of \cite{balmerWittGroups2005}, with duality $\Esc^\vee=\underline{\Hom}_{\Osc_X}(\Esc,\Lc)$ (internal $\Hom$ in the category of $\Osc_X$-modules) and usual biduality. Its Witt group in the sense of Balmer, denoted by $\Wr(X,\Lc)$, controls non-degenerate symmetric bilinear forms with values in $\Lc$ (hereafter $\Lc$-forms) modulo metabolic forms. The tensor product of forms determines an operation $\Wr(X,\Lc)\times\Wr(X,\Lc')\rightarrow\Wr(X,\Lc\otimes\Lc')$ that turns $\Wr(X)=\Wr(X,\Osc_X)$ into a ring and $\Wr(X,\Lc)$ into a $\Wr(X)$-module. Inside the ring $\Wr(X)$ lies the ideal $\Ir(X)$ of even rank forms. We denote by $\Ir^n(X)$ the $n$-th power of this ideal and extend this notation to $n\leqslant 0$ by setting $\Ir^n(X)=\Wr(X)$ in this case. Taking $X=\Spec F$ where $F$ is a field recovers the basic set-up of the theory of symmetric bilinear forms over fields.

Let $U$ be an open subset of $X$. Every invertible section $a\in\Osc_X(U)^\times$ determines a rank $1$ form $\langle a\rangle:x\otimes y\mapsto axy$ on $\Osc_U$, and multiplication by these forms determines an action of $\Osc_X(U)^\times$ on $\Wr(U)$ on the right by group automorphisms that is functorial in $U$. In particular, this action preserves the rank so it induces an action on $\Ir^n(U)$ for every $n\in\Zb$. This construction induces a $\Zb[\Osc_X^\times]$-module structure on the sheaf $\Isc^n$ on $X$ associated with the presheaf $U\mapsto\Ir^n(U)$. We then denote by $\Isc^n(\Lc)$ the twist of $\Isc^n$ by $\Lc$.

The above constructions are functorial with respect to arbitrary morphisms of schemes. In particular, for every open subscheme $U$ of $X$, there is a localisation map $\Ir^n(U)\otimes_{\Zb[\Osc_X(U)^\times]}\Zb[\Lc^0(U)]\rightarrow\prod_{x\in U^{(0)}}\Ir^n(\kappa(x),\Lc(x))$. In fact, assuming that $X\in\Sm_k'$, this map factors through the subgroup $\Ibf_\Lc^n(U)$ of $\bigoplus_{x\in U^{(0)}}\Ir^n(\kappa(x),\Lc(x))$, where we write $\Ibf_\Lc^n$ for $\Ibf^n(\Lc)$ to avoid the cumbersome notation $\Ibf^n(\Lc)(U)$. The resulting homomorphism $\Ir^n(U)\otimes_{\Zb[\Osc_X(U)^\times]}\Zb[\Lc^0(U)]\rightarrow\Ibf_\Lc^n(U)$ is functorial with respect to open immersions, yielding a morphism $\Isc^n(\Lc)\rightarrow\Ibf^n(\Lc)$ of sheaves on $X$ by sheafification.

\begin{theo}
Let $X$ be an essentially smooth $k$-scheme and let $\Lc$ be a line bundle on $X$. Then the morphism $\Isc^n(\Lc)\rightarrow\Ibf^n(\Lc)$ of sheaves on $X$ is an isomorphism.
\end{theo}

\begin{proof}
It suffices to check that this morphism is an isomorphism on stalks. By definition of $\Ibf^n(\Lc)$, it then suffices to show that if $A$ is an essentially smooth local $k$-algebra with fraction field $K$ and $L$ is a line bundle on $\Spec A$, then denoting by $\partial$ the twisted second residue homomorphism for Witt groups, the sequence \[0\rightarrow\Ir^n(A)\otimes_{\Zb[A^\times]}\Zb[L^0]\rightarrow\Ir^n(K,L\otimes_A K)\xrightarrow{\partial}\bigoplus_{x\in(\Spec A)^{(1)}}\Ir^{n-1}(\kappa(x),L(x)\otimes\Lambda_x)\] of abelian groups is exact. This follows, by induction on $n$, from two ingredients (see also \cite[Corollary 7.8]{gilleGradedGerstenWitt2007} for an analogous argument in the context of coherent Witt groups). The first is purity for the Witt group \cite{ojangurenPurityTheoremWitt1999}, namely exactness of the above sequence for Witt groups, thus for $n\leqslant 0$. The second is the affirmation of the universal exactness of the Gersten complex for Milnor $\Kr$-theory in \cite{kerzGerstenConjectureMilnor2009} which implies the exactness of the above complex for $\Kr_*^\Mr/2$ and thus for $\overline{\Ir}^*$ thanks to the affirmation of the Milnor conjecture.
\end{proof}

Thus as in the previous subsection, no distinction between $\Isc^n(\Lc)$ and $\Ibf^n(\Lc)$ is required when dealing with sheaves on $X$. In the sequel, we will use the notation $\Ibf^n(\Lc)$ or $\Ibf_\Lc^n$.

\subsection{The quadratic real cycle class map}

\paragraph*{Jacobson's signature map.} We recall Jacobson's signature morphism of sheaves from \cite{jacobsonRealCohomologyPowers2017}. Let $X$ be a smooth algebraic variety over $\Rb$; denote by $\iota:X(\Rb)\hookrightarrow X$ be the inclusion of the locus $X(\Rb)$ of real points of $X$ into $X$. We endow $X(\Rb)$ with the Euclidean topology for which $X(\Rb)$ underlies a smooth manifold, and with the sheaf $\Csc_{X(\Rb)}$ of continuous functions with values in $\Rb$. Since polynomials are continuous for the Euclidean topology, the map $\iota$ is continuous and in fact underlies a morphism $\iota=(\iota,\iota^\sharp):(X(\Rb),\Csc_{X(\Rb)})\rightarrow(X,\Osc_X)$ of locally ringed spaces, where the morphism $\iota^\sharp:\Osc_X\rightarrow\iota_*\Csc_{X(\Rb)}$ of sheaves restricts the sections of $\Osc_X$ to the real locus. If $\Lc$ is a locally free $\Osc_X$-module of rank $1$, we denote the $\Csc_{X(\Rb)}$-module $\iota^*\Lc$ by $\Lc(\Rb)$ since, geometrically, pulling back along $\iota$ corresponds to taking the real locus.

By normalising the signature $\sign:\Wr(\Rb)\rightarrow\Zb$ of symmetric bilinear forms over $\Rb$, we can define a morphism $\sign_n:\Ibf^n\rightarrow\iota_*\Zb$ of sheaves for every $n\geqslant 1$. Briefly, the morphism $\sign_n$ locally takes $b\in\Ir^n(U)$ to the function $\sign_n(b):x\mapsto\frac{\sign(x^*b)}{2^n}$, where $x^*:\Wr(X)\rightarrow\Wr(\Rb)$ is the restriction to the fibre over the point $x\in U(\Rb)$. This function $\sign_n(b)$ is integer-valued because $b$ lies in the $n$-th power of the fundamental ideal so its signature at each point of $U(\Rb)$ lies in $2^n\Zb$, and is locally constant as noted in \cite[p. 192]{maheSignaturesComposantesConnexes1982} so $\sign_n(b)$ determines a section of $\iota_*\Zb$ on $U$. By construction, multiplication by $\llangle -1\rrangle$ then induces a commutative diagram
\begin{center}
\begin{tikzcd}[column sep = large]
\Wbf=\Ibf^0 \arrow[rrrd,swap,"\sign_0"] \arrow[r,"\otimes\llangle -1\rrangle"] & \Ibf=\Ibf^1 \arrow[rrd,"\sign_1"] \arrow[r,"\otimes\llangle -1\rrangle"] & \cdots \arrow[r,"\otimes\llangle -1\rrangle"] & \Ibf^n \arrow[d,"\sign_n"] \arrow[r,"\otimes\llangle -1\rrangle"] & \cdots \\
 & & & \iota_*\Zb & 
\end{tikzcd}
\end{center}
of sheaves. This diagram induces a morphism $\sign_\infty:\colim\Ibf^n=\Ibf^\infty\rightarrow\iota_*\Zb$ of sheaves on $X$. 

The sheaf $\Osc_X^\times$ of groups acts on $\Ibf^n$ for every $n$ by multiplication by rank $1$ forms, and by commutativity of the product in the Witt ring, multiplication by $\llangle -1\rrangle$ is $\Osc_X^\times$-equivariant for this action. This yields a $\Zb[\Osc_X^\times]$\hbox{-}module structure on the colimit $\Ibf^\infty$. On the other hand, as observed in Example \ref{exe:twist_by_topological_line_bundle}, the sheaf $\Csc_{X(\Rb)}^\times$ acts additively on any abelian sheaf on $X(\Rb)$ by multiplication by the sign of invertible functions. Restricting along the morphism $\iota^\sharp:\Osc_X^\times\rightarrow\iota_*\Csc_{X(\Rb)}^\times$ of sheaves of groups thus induces an additive action of $\Osc_X^\times$ on the pushforward sheaf $\iota_*\Asc$ for every abelian sheaf $\Asc$ on $X(\Rb)$. Now if $U\subseteq X$ is an open subset, then for every $a\in\Osc_X(U)^\times$, the signature of $\langle a\rangle$ and the sign of $\iota^\sharp(a)$ agree as maps from $U(\Rb)$ to $\{\pm 1\}$. It follows that the signature homomorphism $\sign_n:\Ibf^n\rightarrow\iota_*\Zb$ is $\Osc_X^\times$-equivariant for the actions of $\Osc_X^\times$ previously discussed. Hence it can be twisted by $\Lc$ into a map $\sign_n(\Lc):\Ibf^n(\Lc)\rightarrow(\iota_*\Zb)(\Lc)$. These maps for varying $n$ assemble into a morphism $\sign_\infty(\Lc):\colim(\Ibf^n(\Lc))=\Ibf^\infty(\Lc)\rightarrow(\iota_*\Zb)(\Lc)$ of sheaves. Observe that every abelian sheaf $\Asc$ on $X(\Rb)$ can be twisted by $L=\Lc(\Rb)$ into a sheaf $\Asc(L)$ as explained in Example \ref{exe:twist_by_topological_line_bundle}. By \cite[Lemma 2.24]{hornbostelRealCycleClass2021}, we then have an isomorphism $\iota_*(\Asc(L))\cong(\iota_*\Asc)(\Lc)$ of sheaves on $X$ for any $\Asc$ and thus $\iota_*(\Zb(L))\cong(\iota_*\Zb)(\Lc)$ when $\Asc=\Zb$.

\paragraph*{The quadratic real cycle class map for $\Ibf$-cohomology.} We can now define the twisted quadratic real cycle class map $\gamma_\Wr(X,\Lc)$ for $\Ibf$-cohomology. Let again $X$ be a smooth algebraic variety over $\Rb$ and consider the morphism $\iota:(X(\Rb),\Csc_{X(\Rb)})\rightarrow(X,\Osc_X)$ of locally ringed spaces described above. For every $n\geqslant 0$, there is a canonical morphism $i_n:\Ibf^n(\Lc)\rightarrow\Ibf^\infty(\Lc)$ such that $\sign_\infty(\Lc)\circ i_n=\sign_n(\Lc)$. This yields a morphism \[\Hr^c(X,\Ibf^n(\Lc))\xrightarrow{i_n}\Hr^n(X,\Ibf^\infty(\Lc))\xrightarrow{\sign_\infty(\Lc)}\Hr^c(X,\iota_*\Zb(\Lc))\] induced by functoriality in the sheaf of coefficients. We have the following result comparing cohomology on $X$ and cohomology on $X(\Rb)$ for locally constant sheaves on $X(\Rb)$:

\begin{prop}\label{prop:locally_constant_cohomology}
For every locally constant sheaf $\Asc$ on $X(\Rb)$, there is an isomorphism $\Hr^n(X,\iota_*\Asc)\cong\Hr^n(X(\Rb),\Asc)$ of cohomology groups that is functorial with respect to both $X$ and $\Asc$.
\end{prop}

\begin{proof}
Combine \cite[Corollaire 3.7]{costeTopologieSpectreReel1982} and \cite[(19.2) Theorem]{scheidererRealEtaleCohomology1994} (see also \cite[Remark 4.4]{jacobsonRealCohomologyPowers2017}).
\end{proof}

Setting $L=\Lc(\Rb)$, Proposition \ref{prop:locally_constant_cohomology} applies to the locally constant sheaf $\Asc=\Zb(L)$ for which we recall that $\iota_*\Zb(L)\cong(\iota_*\Zb)(\Lc)$. Composing the previous morphism with this isomorphism then yields a morphism \[\gamma_{\Wr,n}^c(X,\Lc):\Hr^c(X,\Ibf^n(\Lc))\rightarrow\Hr^c(X,\Ibf^\infty(\Lc))\xrightarrow{\sign_\infty(\Lc)}\Hr^c(X,\iota_*\Zb(\Lc))\cong\Hr^c(X,(\iota_*\Zb)(L))\cong\Hr^c(X(\Rb),\Zb(L))\] of abelian groups.

\begin{defi}
The map $\gamma_{\Wr,n}^c(X,\Lc):\Hr^c(X,\Ibf^n(\Lc))\rightarrow\Hr^c(X(\Rb),\Zb(L))$ is the quadratic real cycle class map for $\Ibf$-cohomology. When $n=c$, we simply denote it by $\gamma_\Wr^c(X,\Lc):\Hr^c(X,\Ibf^c(\Lc))\rightarrow\Hr^c(X(\Rb),\Zb(L))$. We suppress $\Lc$ from the notation if $\Lc\simeq\Osc_X$.
\end{defi}

In the sequel, we drop the letter $\Wr$ indicating the quadratic (or Witt-theoretic) nature of the preceding construction from the notation since we do not consider the complex and equivariant cycle class maps from the introduction and simply write $\gamma_n^c(X,\Lc)$ (resp. $\gamma^c(X,\Lc)$) for $\gamma_{\Wr,n}^c(X,\Lc)$ (resp. $\gamma_\Wr^c(X,\Lc)$).

\begin{rema}
Our attempt at formulating a conjecture on the image of the quadratic real cycle class maps will be centred on the case of the maps $\gamma^c(X,\Lc):\Hr^c(X,\Ibf^c(\Lc))\rightarrow\Hr^c(X(\Rb),\Zb(L))$ with $n=c$. This is because these diagonal maps induce the quadratic real cycle class map for Chow--Witt groups which can then be compared to the Borel--Haefliger cycle class map using the homomorphism $r:\widetilde{\CH}^c(X,\Lc)\rightarrow\CH^c(X)$.
\end{rema}

The inclusion $\Ibf^{n+1}(\Lc)\rightarrow\Ibf^n(\Lc)$ corresponds to multiplication by $2$ in the sense that the square
\begin{equation}\label{eq:inclusion_(n+1)th_power_into_n}
\begin{tikzcd}
\Ibf^{n+1}(\Lc) \arrow[r] \arrow[d,swap,"\protect{\sign_{n+1}(\Lc)}"] & \Ibf^n(\Lc) \arrow[d,"\protect{\sign_n(\Lc)}"] \\
\iota_*\Zb(\Lc) \arrow[r,swap,"2"] & \iota_*\Zb(\Lc)
\end{tikzcd}
\end{equation}
of sheaves on $X$ commutes. Thus there is an induced map $\overline{\sign}_n(\Lc):\overline{\Ibf}^n\rightarrow\iota_*\Zb(\Lc)/2\iota_*\Zb(\Lc)=\iota_*\Zb/2$ on cokernels (the sheaf $\Osc_X^\times$ acts trivially on $\iota_*\Zb/2$ hence we do not need to take the twist into account). This yields a map $\overline{\gamma}_n^c(X,\Lc):\Hr^c(X,\overline{\Ibf}^n)\rightarrow\Hr^c(X,\iota_*\Zb/2)\cong\Hr^c(X(\Rb),\Zb/2)$ by Proposition \ref{prop:locally_constant_cohomology}. We also set $\overline{\gamma}_c^c(X,\Lc)=\overline{\gamma}^c(X,\Lc)$. By functoriality of sheaf cohomology on $X$ and in view of the natural isomorphism of Proposition \ref{prop:locally_constant_cohomology}, the commutative square of Diagram (\ref{eq:inclusion_(n+1)th_power_into_n}) induces a commutative ladder of abelian groups with exact rows:
\begin{center}
\begin{tikzcd}
\cdots \arrow[r] & \Hr^c(X,\Ibf^{n+1}(\Lc)) \arrow[d,"\protect{\gamma_{n+1}^c(X,\Lc)}"] \arrow[r] & \Hr^c(X,\Ibf^n(\Lc)) \arrow[d,"\protect{\gamma_{n}^c(X,\Lc)}"] \arrow[r] & \Hr^c(X,\overline{\Ibf}^n) \arrow[d,"\protect{\overline{\gamma}_n^c(X,\Lc)}"] \arrow[r] & \cdots \\
\cdots \arrow[r] & \Hr^c(X(\Rb),\Zb(L)) \arrow[r,swap,"2"] & \Hr^c(X(\Rb),\Zb(L)) \arrow[r] & \Hr^c(X(\Rb),\Zb/2) \arrow[r] & \cdots 
\end{tikzcd}
\end{center}
More generally, for every $m\geqslant n$, we have a commutative square
\begin{equation}\label{eq:inclusion_mth_power_into_n}
\begin{tikzcd}
\Ibf^m(\Lc) \arrow[r] \arrow[d,swap,"\protect{\sign_m(\Lc)}"] & \Ibf^n(\Lc) \arrow[d,"\protect{\sign_n(\Lc)}"] \\
\iota_*\Zb(\Lc) \arrow[r,swap,"2^{m-n}"] & \iota_*\Zb(\Lc)
\end{tikzcd}
\end{equation}
and thus an induced homomorphism $\sign_{n/m}(\Lc):\Ibf^n(\Lc)/\Ibf^m(\Lc)\rightarrow\iota_*\Zb/2^{m-n}(\Lc)$ of abelian sheaves on $X$. We denote by $\gamma_{n/m}^c(X,\Lc)$ the morphism \[\gamma_{n/m}^c(X,\Lc):\H^c(X,\Ibf^n(\Lc)/\Ibf^m(\Lc))\rightarrow\Hr^c(X(\Rb),\Zb/2^{m-n}(L))\] of abelian groups induced by $\sign_{n/m}(\Lc)$ using the isomorphism of Proposition \ref{prop:locally_constant_cohomology} for the locally constant sheaf $\Asc=\Zb/2^{m-n}$ on $X(\Rb)$. Thus the map $\gamma_{n/n+1}^c(X,\Lc)$ coincides with the morphism $\overline{\gamma}_n^c(X,\Lc)$. As before, by functoriality of sheaf cohomology on $X$ and in view of Proposition \ref{prop:locally_constant_cohomology}, the above square induces a commutative ladder 
\begin{center}
\begin{tikzcd}
\cdots \arrow[r] & \Hr^c(X,\Ibf^{m}(\Lc)) \arrow[d,"\protect{\gamma_{m}^c(X,\Lc)}"] \arrow[r] & \Hr^c(X,\Ibf^n(\Lc)) \arrow[d,"\protect{\gamma_{n}^c(X,\Lc)}"] \arrow[r] & \Hr^c(X,\Ibf^n(\Lc)/\Ibf^m(\Lc)) \arrow[d,"\protect{\gamma_{n/m}^c(X,\Lc)}"] \arrow[r] & \cdots \\
\cdots \arrow[r] & \Hr^c(X(\Rb),\Zb(L)) \arrow[r,swap,"2^{m-n}"] & \Hr^c(X(\Rb),\Zb(L)) \arrow[r] & \Hr^c(X(\Rb),\Zb/2^{m-n}(L)) \arrow[r] & \cdots 
\end{tikzcd}
\end{center}
of abelian groups with exact rows.

\subparagraph*{Comparison with classical mod $2$ cycle class maps.} Modulo the isomorphism $\overline{\Ibf}^n\cong\Hsc^n$ provided by the affirmation of the Milnor conjecture, the homomorphism $\overline{\sign}_n(\Lc)$ induces a morphism $h_n(\Lc):\Hsc^n\rightarrow\iota_*\Zb/2$ of sheaves on $X$. By the proof of \cite[Proposition 3.21]{hornbostelRealCycleClass2021}, the morphism $h_n(\Lc)$ coincides with the signature mod $2$ map $h_n:\Hsc^n\rightarrow\iota_*\Zb/2$ introduced in \cite{colliot-theleneRealComponentsAlgebraic1990}. In particular, the homomorphism $\overline{\sign}_n(\Lc)$ does not depend on the twisting sheaf $\Lc$ so we omit it from the notation, and by \cite[Remark 2.3.5]{colliot-theleneZerocyclesCohomologyReal1996}, the composite \[\CH^c(X)\rightarrow\CH^c(X)/2\cong\Hr^c(X,\Hsc^c)\cong\Hr^c(X,\overline{\Ibf}^c)\xrightarrow{\overline{\gamma}^c(X)}\Hr^c(X(\Rb),\Zb/2)\] then coincides with $\gamma_\BH^c(X)$. In other words, the diagram
\begin{center}
\begin{tikzcd}
\CH^c(X) \arrow[r] \arrow[rrd,swap,"\gamma_\BH^c(X)"] & \CH^c(X)/2\cong\Hr^c(X,\Hsc^c) \arrow[r,"\cong"] & \Hr^c(X,\overline{\Ibf}^c) \arrow[d,"\overline{\gamma}^c(X)"] & \Hr^c(X,\Ibf^c(\Lc)) \arrow[l] \arrow[d,"\protect{\gamma^c(X,\Lc)}"] \\
                                                      &                                              & \Hr^c(X(\Rb),\Zb/2) & \Hr^c(X(\Rb),\Zb(L)) \arrow[l]
\end{tikzcd}
\end{center}
commutes.

\paragraph*{The quadratic real cycle class map for $\Kbf^\MW$-cohomology.} The morphisms $\gamma_*^*(X,\Lc)$ determine quadratic real cycle class maps for $\Kbf$-cohomology by pre-composition with the morphism $\Kbf_n^\MW(\Lc)\rightarrow\Ibf^n(\Lc)$. In other words, we have a homomorphism \[\widetilde{\gamma}_{\Rb,n}^c(X,\Lc):\Hr^c(X,\Kbf_n^\MW(\Lc))\rightarrow\Hr^c(X,\Ibf^n(\Lc))\xrightarrow{\gamma_n^c(X,\Lc)}\Hr^c(X(\Rb),\Zb(L))\] of abelian groups. In the sequel, we drop the index $\Rb$ as, again, we do not consider the complex and equivariant maps from the introduction. We are particularly interested in the case $n=c$ for its relevance for Chow--Witt groups: this is then a homomorphism $\widetilde{\gamma}^c(X,\Lc):\widetilde{\CH}^c(X,\Lc)\rightarrow\Hr^c(X(\Rb),\Zb(L))$ of abelian groups. The square
\begin{center}
\begin{tikzcd}
\widetilde{\CH}^c(X,\Lc) \arrow[r,"r"] \arrow[d,swap,"\protect{\widetilde{\gamma}^c(X,\Lc)}"] & \CH^c(X) \arrow[d,"\gamma_\BH^c(X)"] \\
\Hr^c(X(\Rb),\Zb(L)) \arrow[r] & \Hr^c(X(\Rb),\Zb/2)
\end{tikzcd}
\end{center}
whose bottom horizontal map is reduction mod $2$ of the coefficients commutes hence the morphism $\widetilde{\gamma}^c(X,\Lc)$ refines the Borel--Haefliger cycle class map $\gamma_\BH^c(X)$.

\begin{lem}\label{lem:chow_witt_i_cohomology_cycle_map_same_image}
Let $c\geqslant n\geqslant 0$ be integers. Then the homomorphisms $\widetilde{\gamma}_{n}^c(X,\Lc)$ and $\gamma_{n}^c(X,\Lc)$ have the same image in $\Hr^c(X(\Rb),\Zb(L))$.
\end{lem}

This applies in particular to the quadratic real cycle class maps $\widetilde{\gamma}^c(X,\Lc)$ and $\gamma^c(X,\Lc)$.

\begin{proof}
By definition, we have a commutative triangle
\begin{center}
\begin{tikzcd}
\Hr^c(X,\Kbf_n^\MW(\Lc)) \arrow[rd,swap,"\protect{\widetilde{\gamma}_n^c(X,\Lc)}"] \arrow[r,"u"] & \Hr^c(X,\Ibf^n(\Lc)) \arrow[d,"\protect{\gamma_n^c(X,\Lc)}"] \\
& \Hr^c(X(\Rb),\Zb(L))
\end{tikzcd}
\end{center}
where $u:\Hr^c(X,\Kbf_n^\MW(\Lc))\rightarrow\Hr^c(X,\Ibf^n(\Lc))$ is induced by the structure morphism $\Kbf_n^\MW(\Lc)\rightarrow\Ibf^n(\Lc)$ of sheaves. Since $u$ is an epimorphism by Lemma \ref{lem:milnor_k_theory_to_i_cohomology_surjective}, the result follows.
\end{proof}

\paragraph*{Relevant results on the quadratic real cycle class map.} In this paragraph, we let $X$ be a smooth $\Rb$-variety of dimension $d$ and $\Lc$ be a line bundle on $X$, and we set $L=\Lc(\Rb)$.

\begin{theo}[Jacobson]\label{theo:twisted_signature_iso of sheaves}
The twisted signature map $\sign_\infty(\Lc):\Ibf^\infty(\Lc)\rightarrow\iota_*\Zb(\Lc)$ is an isomorphism.
\end{theo}

\begin{proof}
The twist of an isomorphism of sheaves being an isomorphism, it suffices to consider the untwisted case which is \cite[Theorem 8.6 (i)]{jacobsonRealCohomologyPowers2017}, see \cite[Theorem 3.10 (1)]{hornbostelRealCycleClass2021}.
\end{proof}

\begin{theo}\label{theo:stable_range_isomorphism}
Let $n$ be an integer such that $n\geqslant d+1$. Then multiplication by $\llangle -1\rrangle$ determines an isomorphism $\otimes\llangle-1\rrangle:\Ibf^{n}(\Lc)\rightarrow\Ibf^{n+1}(\Lc)$ of sheaves on $X$.
\end{theo}

This is essentially the proof of \cite[Corollary 8.11]{jacobsonRealCohomologyPowers2017} (that covers the untwisted case, from which the twisted case immediately follows). We will need a more refined statement so we quote the relevant result from the theory of Witt rings of extensions of $\Rb$.

\begin{prop}\label{prop:stable_range_isomorphism_witt_groups_of_fields}
Let $F/\Rb$ be a finitely generated extension of transcendence degree $d$. Then multiplication by $\llangle -1\rrangle$ induces an isomorphism $\Ir^n(F)\rightarrow\Ir^{n+1}(F)$ for every $n\geqslant d+1$ and an epimorphism $\Ir^d(F)\rightarrow\Ir^{d+1}(F)$.
\end{prop}

\begin{proof}
If $-1$ is a square in $F$, then $F$ can be viewed as a finitely generated extension of $\Cb$ of transcendence degree $d$. By \cite[Theorem 97.7]{elmanAlgebraicGeometricTheory2008}, the field $F$ is then a $C_d$-field in the sense of \cite[Section 97]{elmanAlgebraicGeometricTheory2008}: this means that every homogeneous polynomial of degree $n$ in $m>n^d$ variables vanishes at a point in $F^{m}$ distinct from $(0,\ldots,0)$. In particular, every quadratic form in $>2^d$ variables with coefficients in $F$ is isotropic, namely has a nontrivial zero. This applies to the $(d+1)$-fold Pfister forms $\llangle a_1\rrangle\otimes\cdots\otimes\llangle a_{d+1}\rrangle$ where each $a_i$ is a unit in $F$, which have dimension $2^{d+1}>2^d$. Now by \cite[Corollary 9.11]{elmanAlgebraicGeometricTheory2008}, isotropic Pfister forms are hyperbolic, namely have zero Witt class in $\Wr(F)$. By \cite[Corollary 4.9]{elmanAlgebraicGeometricTheory2008}, the subgroup $\Ir^{d+1}(F)$ of $\Wr(F)$ is generated by $(d+1)$-fold Pfister forms: we conclude that the group $\Ir^{d+1}(F)$ vanishes. Therefore the map $\otimes\llangle-1\rrangle:\Ir^d(F)\rightarrow\Ir^{d+1}(F)=0$ is surjective. Since $\Ir^n(F)\subseteq\Ir^{d+1}(F)$ for every $n\geqslant d+1$, the group $\Ir^n(F)$ also vanishes for every such $n$, hence multiplication by $\llangle -1\rrangle$ is an isomorphism from $0=\Ir^n(F)$ to $0=\Ir^{n+1}(F)$.

Now assume that $-1$ is not a square in $F$ and consider the field $K=F[\sqrt{-1}]$. Then the extension $K/\Cb$ is finitely generated of transcendence degree $d$ so again by \cite[Theorem 97.7]{elmanAlgebraicGeometricTheory2008}, the field $K$ is a $C_d$-field. Consequently, by the argument of the previous paragraph, the group $\Ir^{d+1}(K)$ vanishes and thus so does its subgroup $\Ir^{n+1}(K)$ for every $n\geqslant d$. We then deduce from \cite[Corollary 35.27]{elmanAlgebraicGeometricTheory2008} that for every $n\geqslant d$, the equality $2\Ir^n(F)=\Ir^{n+1}(F)$ holds in $\Wr(F)$ and the group $\Ir^{n+1}(F)$ is torsion free, so that the map $\cdot 2:\Ir^{n+1}(F)\rightarrow\Ir^{n+2}(F)$ is injective. Since $2=\langle 1,1\rangle=\llangle -1\rrangle$ in the Witt ring, we conclude that $\otimes\llangle -1\rrangle:\Ir^d(F)\rightarrow\Ir^{d+1}(F)$ is surjective and that $\otimes\llangle -1\rrangle:\Ir^{n+1}(F)\rightarrow\Ir^{n+2}(F)$ is an isomorphism of abelian groups for every $n\geqslant d$, as required.
\end{proof}

\begin{cor}[\protect{\cite[Corollary 8.11]{jacobsonRealCohomologyPowers2017}}]\label{cor:signature_iso_of_sheaves_finite_range}
Let $n\geqslant d+1$. Then the twisted signature map $\sign_{n}(\Lc):\Ibf^{n}(\Lc)\rightarrow\iota_*\Zb(\Lc)$ is an isomorphism of sheaves on $X$.
\end{cor}

\begin{proof}
Write the sheaf $\Ibf^\infty(\Lc)$ as $\Ibf^\infty(\Lc)=\colim_{m\geqslant d+1}\Ibf^{m}(\Lc)$. By Theorem \ref{theo:stable_range_isomorphism}, the transition morphisms of this colimit are isomorphisms. Therefore for every $n\geqslant d+1$, the canonical map $i_{n}:\Ibf^{n}(\Lc)\rightarrow\Ibf^\infty(\Lc)$ is an isomorphism. Then $\sign_{n}(\Lc)=\sign_\infty(\Lc)\circ i_n$ is a composition of isomorphisms by Jacobson's theorem \ref{theo:twisted_signature_iso of sheaves} and is consequently an isomorphism.
\end{proof}

By Theorem \ref{theo:twisted_signature_iso of sheaves}, for every $c$ and every $n$, in the composition \[\gamma_{n}^c(X,\Lc):\Hr^c(X,\Ibf^n(\Lc))\xrightarrow{i_n}\Hr^c(X,\Ibf^\infty(\Lc))\xrightarrow{\sign_\infty(\Lc)}\Hr^c(X,\iota_*\Zb(\Lc))\xrightarrow{\cong}\Hr^c(X(\Rb),\Zb(L)),\] each map except possibly for the first map induced by the canonical morphism $i_n:\Ibf^n(\Lc)\rightarrow\Ibf^\infty(\Lc)$ of sheaves is a group isomorphism. Thus the study of the morphism $\gamma_{n}^c(X,\Lc)$ reduces to that of the morphism $\Hr^c(X,\Ibf^n(\Lc))\rightarrow\Hr^c(X,\Ibf^\infty(\Lc))$ and is consequently a purely algebraic problem. We also note that by Corollary \ref{cor:signature_iso_of_sheaves_finite_range}, we do not need to go to infinity to reach the colimit $\Ibf^\infty(\Lc)$. Indeed, if $n\leqslant d$, which is the relevant range for the quadratic real cycle class map for Chow--Witt groups, then multiplication by $\llangle -1\rrangle^{\otimes(d+1-n)}$ from $\Ibf^n(\Lc)$ to $\Ibf^{d+1}(\Lc)$ determines a commutative triangle:
\begin{center}
\begin{tikzcd}[column sep=huge]
\Hr^c(X,\Ibf^n(\Lc)) \arrow[r] \arrow[rd,swap,"\protect{\gamma_{n}^c(X,\Lc)}"] & \Hr^c(X,\Ibf^{d+1}(\Lc)) \arrow[d,"\protect{\gamma_{d+1}^c(X,\Lc)}"] \\
& \Hr^c(X(\Rb),\Zb(L))
\end{tikzcd}
\end{center}
in which the map $\gamma_{d+1}^c(X,\Lc)$ is an isomorphism. This further reduces the study of the maps $\gamma_n^c(X,\Lc)$ for $n\leqslant d$ to that of the maps induced in cohomology by the morphism $\otimes\llangle -1\rrangle^{\otimes(d+1-n)}:\Ibf^n(\Lc)\rightarrow\Ibf^{d+1}(\Lc)$ of sheaves.

\begin{rema}\label{rema:morphism_of_homotopy_modules}
Let $m\geqslant 0$ be an integer. The morphisms $\otimes\llangle -1\rrangle^{\otimes m}:\Ibf^n\rightarrow\Ibf^{n+m}$ assemble into a morphism $\otimes\llangle -1\rrangle^{\otimes m}:\Ibf^*\rightarrow\Ibf^{*+m}$ of graded sheaves. In fact, it is a morphism of homotopy modules, namely, it is compatible with the structure isomorphisms $\Ibf^{n-1}\cong(\Ibf^n)_{-1}$ of sheaves: this is because, since $\llangle -1\rrangle=2$ in the Witt ring, the morphism $\otimes\llangle -1\rrangle^{\otimes m}$ is multiplication by the integer $2^m$, and the structure isomorphisms are morphisms of abelian sheaves. Hence by Proposition \ref{prop:category_of_homotopy_modules}, the kernels $\Kbf_n^m$, the images $\Jbf_n^m$ and the cokernels $\Qbf_n^m$ of the homomorphisms $\otimes\llangle -1\rrangle^{\otimes m}:\Ibf^n\rightarrow\Ibf^{n+m}$ assemble into homotopy modules $\Kbf_*^m$, $\Qbf_*^m$ and $\Jbf_*^m$ respectively. In particular, the component sheaves of these homotopy modules are \emph{contractions}. We conclude that for every $X\in\Sm_\Rb$ and every line bundle $\Lc$ on $X$, the sheaves $\Kbf_n^m$, $\Jbf_n^m$ and $\Qbf_n^m$ can be twisted by $\Lc$ and the resulting twisted sheaves have twisted Rost--Schmid complexes. We use these facts implicitly in the sequel.
\end{rema}

\section{The quadratic real cycle class map for zero-cycles}\label{section:closed_points}

Let $X$ be a connected smooth real algebraic variety of dimension $d$ and let $\Lc$ be a locally free $\Osc_X$-module of rank $1$. We set $L=\Lc(\Rb)$. In this section, we study the homomorphism \[\gamma^d(X,\Lc):\Hr^d(X,\Ibf^d(\Lc))\rightarrow\Hr^d(X(\Rb),\Zb(L))\] of abelian groups. Multiplication by $\llangle -1\rrangle$ determines a morphism $\otimes\llangle -1\rrangle:\Ibf^d(\Lc)\rightarrow\Ibf^{d+1}(\Lc)$ and for every $c\geqslant 0$, denoting by $f:\Hr^c(X,\Ibf^d(\Lc))\rightarrow\Hr^c(X,\Ibf^{d+1}(\Lc))$ induced by $\otimes\llangle -1\rrangle$ in cohomology, the triangle
\begin{center}
\begin{tikzcd}
\Hr^c(X,\Ibf^d(\Lc)) \arrow[r,"f"] \arrow[rd,swap,"\protect{\gamma_d^c(X,\Lc)}"] & \Hr^c(X,\Ibf^{d+1}(\Lc)) \arrow[d,"\protect{\gamma_{d+1}^c(X,\Lc)}"] \\
                                                               & \Hr^c(X(\Rb),\Zb(L))
\end{tikzcd}
\end{center}
commutes. Furthermore, as explained at the end of the last section, the vertical map $\gamma_{d+1}^c(X,\Lc)$ in this triangle is an isomorphism. This reduces the study of $\gamma_d^c(X,\Lc)$ to the analysis of the horizontal arrow $f$, which is purely algebraic. In the sequel, we use this reduction freely and essentially do not distinguish these two homomorphisms $\gamma_{d}^c(X,\Lc)$ and $f$.

\begin{lem}\label{lem:multiplication_by_-1_epimorphism}
The map $\otimes\llangle -1\rrangle:\Ibf^d(\Lc)\rightarrow\Ibf^{d+1}(\Lc)$ of sheaves on $X$ is an epimorphism.
\end{lem}

\begin{proof}
The cokernel $\Qbf(\Lc)$ of this morphism of sheaves is the twist $\Qbf_d^1(\Lc)$ of the sheaf $\Qbf=\Qbf_d^1$ of Remark \ref{rema:morphism_of_homotopy_modules} and thus has a Rost--Schmid complex. Therefore to establish that $\Qbf(\Lc)$ is the zero sheaf, it suffices to prove that $\Qr(K,\Lc\otimes K)=0$ where $K$ is the function field of $X$. The group $\Qr(K,\Lc\otimes K)$ is (non-canonically) isomorphic to $\Qr(K)$ so it suffices to show that the group $\Qr(K)$ vanishes. As we observed in Remark \ref{rema:colimit_extension_essentially_smooth}, this group is the cokernel of the homomorphism $\otimes\llangle -1\rrangle:\Ir^d(K)\rightarrow\Ir^{d+1}(K)$ of abelian groups. Since $X$ is of dimension $d$, the field $K$ is of transcendence degree $d$ over $\Rb$ thus the morphism $\otimes\llangle -1\rrangle:\Ir^d(K)\rightarrow\Ir^{d+1}(K)$ is surjective by Proposition \ref{prop:stable_range_isomorphism_witt_groups_of_fields}. Consequently, its cokernel $\Qr(K)$ vanishes as required.
\end{proof}

Let $\Kbf(\Lc)$ be the kernel of the morphism $\otimes\llangle -1\rrangle:\Ibf^d(\Lc)\rightarrow\Ibf^{d+1}(\Lc)$, so that we have an exact sequence \[0\rightarrow\Kbf(\Lc)\rightarrow\Ibf^d(\Lc)\xrightarrow{\otimes\llangle -1\rrangle}\Ibf^{d+1}(\Lc)\rightarrow 0\] of sheaves on $X$. We note that multiplication by $\llangle -1\rrangle$ sends $\Ibf^{d+1}(\Lc)\subseteq\Ibf^{d}(\Lc)$ into $\Ibf^{d+2}(\Lc)\subseteq\Ibf^{d+1}(\Lc)$ so there is an induced morphism $\otimes\llangle -1\rrangle:\overline{\Ibf}^d\rightarrow\overline{\Ibf}^{d+1}$ that fits in a commutative ladder with exact rows:
\begin{equation}\label{eq:commutative_ladder_I_to_overline_I}
\begin{tikzcd}[column sep=large]
0 \arrow[r] & \Kbf(\Lc) \arrow[r] \arrow[d,"u"] & \Ibf^d(\Lc) \arrow[r,"\otimes\llangle -1\rrangle"] \arrow[d] & \Ibf^{d+1}(\Lc) \arrow[r] \arrow[d] & 0 \\
0 \arrow[r] & \overline{\Kbf} \arrow[r] & \overline{\Ibf}^d \arrow[r,swap,"\otimes\llangle -1\rrangle"] & \overline{\Ibf}^{d+1} \arrow[r] & 0
\end{tikzcd}
\end{equation}
where the two right vertical morphisms are the quotient maps and $u$ is induced by the right commutative square.

\begin{lem}\label{lem:kernel_as_torsion}
The morphism $u:\Kbf(\Lc)\rightarrow\overline{\Kbf}$ is an isomorphism of sheaves on $X$. In particular, the sheaf $\Kbf(\Lc)$ is independent from $\Lc$.
\end{lem}

\begin{proof}
Let $\Abf(\Lc)$ be the kernel of $u$ and $\Bbf(\Lc)$ be its cokernel. Then we have the following commutative diagram:
\begin{center}
\begin{tikzcd}
            & 0 \arrow[d]                   & 0 \arrow[d] & 0 \arrow[d] & \\
            & \Abf(\Lc) \arrow[d] & \Ibf^{d+1}(\Lc) \arrow[d] & \Ibf^{d+2}(\Lc) \arrow[d] & \\ 
0 \arrow[r] & \Kbf(\Lc) \arrow[r] \arrow[d,"u"] & \Ibf^d(\Lc) \arrow[r,"\otimes\llangle -1\rrangle"] \arrow[d] & \Ibf^{d+1}(\Lc) \arrow[r] \arrow[d] & 0 \\
0 \arrow[r] & \overline{\Kbf} \arrow[r] \arrow[d] & \overline{\Ibf}^d \arrow[r,swap,"\otimes\llangle -1\rrangle"] \arrow[d] & \overline{\Ibf}^{d+1} \arrow[r] \arrow[d] & 0 \\
            & \Bbf(\Lc) \arrow[d] & 0 & 0 & \\
            & 0                 &   &   &
\end{tikzcd}
\end{center}
with exact rows and columns. The snake lemma then provides an exact sequence \[0\rightarrow\Abf(\Lc)\rightarrow\Ibf^{d+1}(\Lc)\xrightarrow{\otimes\llangle -1\rrangle}\Ibf^{d+2}(\Lc)\rightarrow\Bbf(\Lc)\rightarrow 0\] of sheaves on $X$ where the morphism $\otimes\llangle -1\rrangle$ is an isomorphism by Theorem \ref{theo:stable_range_isomorphism}. Consequently, the sheaves $\Abf(\Lc)$ and $\Bbf(\Lc)$ vanish so $u$ is an isomorphism as required.
\end{proof}

Lemmas \ref{lem:multiplication_by_-1_epimorphism} and \ref{lem:kernel_as_torsion} thus show that there is a short exact sequence 
\begin{equation}\label{eq:exact_sequence_kernel_-1_i^d}
0\rightarrow\overline{\Kbf}\rightarrow\Ibf^d(\Lc)\xrightarrow{\otimes\llangle -1\rrangle}\Ibf^{d+1}(\Lc)\rightarrow 0
\end{equation} of sheaves on $X$. Moreover, Lemma \ref{lem:kernel_as_torsion} allows us to use the results of \cite{colliot-theleneZerocyclesCohomologyReal1996}. More precisely, recall that the affirmation of the Milnor conjecture yields an isomorphism $\overline{\Ibf}^*\cong\mathscr{H}^*$ of sheaves of graded rings on $X$. This isomorphism is given by $\llangle a\rrangle\mapsto(a)$ on fields in degree $1$: in particular, multiplication by $\llangle -1\rrangle$ from $\overline{\Ibf}^n$ to $\overline{\Ibf}^{n+1}$ is identified with cup-product with $(-1)$ from $\mathscr{H}^n$ to $\mathscr{H}^{n+1}$ by this isomorphism. Thus the sheaf $\overline{\Kbf}$ of Lemma \ref{lem:kernel_as_torsion} can be identified with the kernel of the map $(-1)\cup\text{--}:\mathscr{H}^d\rightarrow\mathscr{H}^{d+1}$ of \cite[Equation (2.6)]{colliot-theleneZerocyclesCohomologyReal1996}.

\begin{prop}\label{prop:injectivity_real_cycle_class_map_closed_points}
We have $\Hr^d(X,\overline{\Kbf})=0$ if $X$ is not proper or $X(\Rb)$ is non-empty.
\end{prop}

This is \cite[Theorem 3.2 (d)]{colliot-theleneZerocyclesCohomologyReal1996} but for the convenience of the reader, we recall the proof. It makes use of the \emph{transfer} along finite étale maps. Set $X'=X\times_\Rb\Spec\Cb$; we still denote the Zariski sheafification of the presheaf $V\mapsto\Hr_\et^d(V,\Zb/2)$ on $X'$ by $\Hsc^d$. Let $p:X'\rightarrow X$ be the projection. Then $p$ is finite étale of degree $2$, and the pullback and transfer along $p$ induce homomorphisms $p^*:\Hsc^d\rightarrow p_*\Hsc^d$ and $p_*:p_*\Hsc^d\rightarrow\Hsc^d$ of sheaves on $X$ such that the sequence
\begin{equation}\label{eq:exact_sequence_étale_cohomology_sheaves}
\Hsc^d\xrightarrow{p^*}p_*\Hsc^d\xrightarrow{p_*}\Hsc^d\xrightarrow{(-1)\cup\text{--}}\Hsc^{d+1}
\end{equation}
is exact \cite[Lemma 2.2.1]{colliot-theleneZerocyclesCohomologyReal1996}. Moreover, for every $j\geqslant 1$, the stalk of the higher direct image $\Rr^jp_*\Hsc^d$ at the point $x\in X$ is $\Hr^j(Y,\Hsc^d)$ where $Y$ is the semi-localisation of $X'$ at the finite set $p^{-1}(x)$. By Bloch--Ogus' acyclicity result for semi-local schemes cited in the proof of Theorem \ref{theo:unramified_étale_cohomology}, the group $\Hr^j(Y,\Hsc^d)$ vanishes so $\Rr^jp_*\Hsc^d=0$ and it follows that we have a natural isomorphism $\Hr^c(X,p_*\Hsc^d)=\Hr^c(X',\Hsc^d)$ for every $c\geqslant 0$.

\begin{proof}
The exact sequence (\ref{eq:exact_sequence_kernel_-1_i^d}) of sheaves on $X$ yields an exact sequence \[\Hr^{d-1}(X,\Hsc^d)\rightarrow\Hr^{d-1}(X,\Hsc^{d+1})\rightarrow\Hr^d(X,\overline{\Kbf})\rightarrow\Hr^d(X,\Hsc^d)\rightarrow\Hr^d(X,\Hsc^{d+1})\] where the first map is surjective and the last map is injective, in both cases owing to \cite[Theorem 3.2 (d)]{colliot-theleneZerocyclesCohomologyReal1996}. This implies the vanishing of the group $\Hr^d(X,\overline{\Kbf})$. For the convenience of the reader, we recall the proof of \emph{loc. cit.} In view of the exact sequence (\ref{eq:exact_sequence_étale_cohomology_sheaves}) of sheaves, the transfer along $p$ induces an epimorphism $p_*:p_*\Hsc^d\rightarrow\overline{\Kbf}$ of sheaves on $X$. Let $\overline{\Qbf}$ denote the kernel of $p_*$. The map $\Hr^d(X',\Hsc^d)\rightarrow\Hr^d(X,\overline{\Kbf})$ induced by the transfer is surjective: indeed, its cokernel is a subgroup of $\Hr^{d+1}(X,\overline{\Qbf})$ and the scheme $X$ is of Zariski cohomological dimension $\leqslant d$ since $X$ is Noetherian (Example 2.2). By the same argument, the epimorphism $p^*:\Hsc^d\rightarrow\overline{\Qbf}$ of sheaves induces an epimorphism $\Hr^d(X,\Hsc^d)\rightarrow\Hr^d(X,\overline{\Qbf})$ of abelian groups, and the composite \[\Hr^d(X,\Hsc^d)\rightarrow\Hr^d(X,\overline{\Qbf})\rightarrow\Hr^d(X',\Hsc^d)\] is identified by the natural isomorphisms $\Hr^d(X,\Hsc^d)\cong\CH^d(X)/2$ and $\Hr^d(X',\Hsc^d)\cong\CH^d(X')/2$ with the pullback homomorphism $p^*:\CH^d(X)/2\rightarrow\CH^d(X')/2$ for the flat morphism $p$. Therefore the cohomology long exact sequence for the epimorphism $p_*\Hsc^d\rightarrow\overline{\Kbf}$ induces an exact sequence \[\CH^d(X)/2\xrightarrow{p^*}\CH^d(X')/2\rightarrow\Hr^d(X,\overline{\Kbf})\rightarrow 0\] of abelian groups. Thus to show that $\Hr^d(X,\overline{\Kbf})$ vanishes under the stated hypotheses, it suffices to prove that $p^*:\CH^d(X)/2\rightarrow\CH^d(X')/2$ is surjective if $X$ is not proper or $X(\Rb)$ is non-empty.

Assume that $X$ is not proper. Then by \cite[\href{https://stacks.math.columbia.edu/tag/03GN}{Tag 03GN}]{stacks-project}, the $\Cb$-scheme $X'$ is not proper. The scheme $X'$ is either connected; or it is disconnected, which means that $X$ is obtained as the pushforward of a $\Cb$-scheme $Y$ along the map $\Spec\Cb\rightarrow\Spec\Rb$ and $X'$ is the disjoint union of two copies of $Y$ (\cite[Lemma 1.1]{colliot-theleneZerocyclesCohomologyReal1996}). In any case, the $\Cb$-scheme $X'$ does not have a proper connected component: by \cite[Lemma 1.2]{colliot-theleneZerocyclesCohomologyReal1996}, the group $\CH^d(X')$ is divisible hence $\CH^d(X')/2=0$ and the surjectivity of the map $p^*:\CH^d(X)/2\rightarrow\CH^d(X')/2$ is then immediate. Now assume that $X$ is proper and that $X(\Rb)$ is not empty; in particular, the scheme $X'$ is then connected. Let $x\in X(\Rb)$ and let $x'$ be the unique point of $X'$ lying over $x$. Since $X'$ is proper, the structure map $X'\rightarrow\Spec\Cb$ induces a morphism $\CH^d(X')\rightarrow\CH^0(\Spec\Cb)=\Zb$ which sends the class $[x']$ of $x'$ to $1$. According to \cite[Lemma 1.2]{colliot-theleneZerocyclesCohomologyReal1996}, the kernel of this morphism is a divisible group, therefore the class $[x']$ generates $\CH^d(X')/2$. But since $p$ is étale, the pullback homomorphism $p^*:\CH^d(X)\rightarrow\CH^d(X')$ carries the class $[x]$ of $x$ to $[x']$. In particular, the morphism $p^*:\CH^d(X)/2\rightarrow\CH^d(X')/2$ is surjective as claimed.
\end{proof}

\begin{rema}
A similar argument can be carried out to show more directly that $\Hr^d(X,\Kbf(\Lc))=0$, using the fact that the transfer along the projection map $p:X'\rightarrow X$ exhibits $\Kbf(\Lc)$ as a quotient of $p_*\Ibf^d(\Lc)$.
\end{rema}

\begin{theo}\label{theo:conjecture_for_closed_points}
The group homomorphism $\gamma^d(X,\Lc)$ is surjective. It is injective if either $X$ is not proper or $X(\Rb)$ is non-empty.
\end{theo}

\begin{proof}
The short exact sequence \[0\rightarrow\overline{\Kbf}\rightarrow\Ibf^d(\Lc)\xrightarrow{\otimes\llangle -1\rrangle}\Ibf^{d+1}(\Lc)\rightarrow 0\] of sheaves (\ref{eq:exact_sequence_kernel_-1_i^d}) deduced from Lemmas \ref{lem:multiplication_by_-1_epimorphism} and Lemma \ref{lem:kernel_as_torsion} induces an exact sequence \[\Hr^d(X,\overline{\Kbf})\rightarrow\Hr^d(X,\Ibf^d(\Lc))\xrightarrow{f}\Hr^d(X,\Ibf^{d+1}(\Lc))\rightarrow\Hr^{d+1}(X,\overline{\Kbf})\] in cohomology. In this exact sequence, as noted at the beginning of this section, the morphism $f$ induced by the map $\otimes\llangle -1\rrangle$ is identified with the morphism $\gamma^d(X,\Lc)$ by the isomorphism $\gamma_{d+1}^d(X,\Lc)$. Thus it suffices to show that $f$ is surjective, and that $f$ is injective if $X$ is not proper or $X(\Rb)\neq\emptyset$. The statement regarding injectivity is a direct consequence of the vanishing result of Proposition \ref{prop:injectivity_real_cycle_class_map_closed_points}. On the other hand, the scheme $X$ is of Zariski cohomological dimension $\leqslant d$ hence the group $\Hr^{d+1}(X,\overline{\Kbf})$ vanishes and thus $f$ is surjective as required. 
\end{proof}

\begin{rema}
Theorem \ref{theo:conjecture_for_closed_points} appears, with essentially the same proof, in \cite[Proposition 3.4]{hornbostelFewComputationsReal2024} (in which \og if $X$ is affine \fg{} may be replaced by \og if $X$ is not proper \fg{} in view of the proof of the proposition given by Hornbostel).
\end{rema}

\begin{rema}
Assume that $X(\Rb)=\emptyset$. Then $\Hr^0(U,\Ibf^{d+1}(\Lc))\cong\Hr^0(U(\Rb),\Zb(L))$ vanishes for every open subset $U$ of $X$ hence the sheaf $\Ibf^{d+1}(\Lc)$ on $X$ vanishes. We deduce that the quotient map $\Ibf^d(\Lc)\rightarrow\overline{\Ibf}^d$ is an isomorphism of sheaves, so that we have a string \[\Hr^d(X,\Ibf^d(\Lc))\xrightarrow{\cong}\Hr^d(X,\overline{\Ibf}^d)\cong\Hr^d(X,\Hsc^d)\cong\CH^d(X)/2\] of isomorphisms of abelian groups. If moreover $X$ is proper, then the group $\CH^d(X)/2$ is isomorphic to $\Zb/2$ by \cite[Theorem 1.3 (b)]{colliot-theleneZerocyclesCohomologyReal1996}: in particular, the map $\Zb/2=\Hr^d(X,\Ibf^d(\Lc))\rightarrow\Hr^d(X(\Rb),\Zb(L))=0$ is not injective. Consequently, the reverse implication in the last assertion of Theorem \ref{theo:conjecture_for_closed_points} is also true.
\end{rema}

\section{Descending exponent bounds}

\subsection{General framework}

As before, we let $X$ be a connected smooth algebraic variety over $\Rb$ and $\Lc$ be a line bundle on $X$, and we set $L=\Lc(\Rb)$. In this subsection, we explain how to exploit exponent results for cohomology in higher powers of the fundamental ideal to get exponent bounds for the range in which we are investigating the image of the quadratic real cycle class map. The general framework for results of this type is provided by the following lemma.

\begin{lem}\label{lem:descending_exponents}
Let $c$ and $m\geqslant n$ be non-negative integers. Then we have inclusions \[2^{m-n}\Im\gamma_m^c(X,\Lc)\subseteq\Im\gamma_n^c(X,\Lc)\subseteq\Im\gamma_m^c(X,\Lc)\subseteq\Hr^c(X(\Rb),\Zb(L)).\] Consequently, the group $\Im\gamma_m^c(X,\Lc)/\Im\gamma_n^c(X,\Lc)$ has exponent $2^{m-n}$; in particular, if the morphism $\gamma_m^c(X,\Lc)$ is surjective, then the cokernel of $\gamma_n^c(X,\Lc)$ has exponent $2^{m-n}$.
\end{lem}

\begin{proof}
By construction, the morphism $\otimes\llangle -1\rrangle^{\otimes(m-n)}:\Ibf^n(\Lc)\rightarrow\Ibf^m(\Lc)$ of sheaves induces a commutative triangle
\begin{center}
\begin{tikzcd}[column sep=huge]
\Hr^c(X,\Ibf^n(\Lc)) \arrow[r] \arrow[rd,swap,"\protect{\gamma_{n}^c(X,\Lc)}"] & \Hr^c(X,\Ibf^{m}(\Lc)) \arrow[d,"\protect{\gamma_{m}^c(X,\Lc)}"] \\
& \Hr^c(X(\Rb),\Zb(L))
\end{tikzcd}
\end{center}
In other words, the morphism $\gamma_n^c(X,\Lc)$ factors through the map $\gamma_m^c(X,\Lc)$: this proves that $\Im\gamma_n^c(X,\Lc)\subseteq\Im\gamma_m^c(X,\Lc)$. On the other hand, recall the commutative square
\begin{center}
\begin{tikzcd}
\Ibf^m(\Lc) \arrow[r] \arrow[d,swap,"\protect{\sign_m(\Lc)}"] & \Ibf^n(\Lc) \arrow[d,"\protect{\sign_n(\Lc)}"] \\
\iota_*\Zb(\Lc) \arrow[r,swap,"2^{m-n}"]                       & \iota_*\Zb(\Lc)
\end{tikzcd}
\end{center}
of sheaves on $X$ (Diagram (\ref{eq:inclusion_mth_power_into_n})). This square induces a commutative square
\begin{center}
\begin{tikzcd}
\Hr^c(X,\Ibf^m(\Lc)) \arrow[r] \arrow[d,swap,"\protect{\gamma_{m}^c(X,\Lc)}"] & \Hr^c(X,\Ibf^n(\Lc)) \arrow[d,"\protect{\gamma_{n}^c(X,\Lc)}"] \\
\Hr^c(X(\Rb),\Zb(L)) \arrow[r,swap,"2^{m-n}"] & \Hr^c(X(\Rb),\Zb(L))
\end{tikzcd}
\end{center}
of abelian groups. It follows that the image of the composite map \[\Hr^c(X,\Ibf^m(\Lc))\rightarrow\Hr^c(X,\Ibf^n(\Lc))\xrightarrow{\gamma_{n}^c(X,\Lc)}\Hr^c(X(\Rb),\Zb(L))\] contains $2^{m-n}\Im\gamma_{m}^c(X,\Lc)$: \emph{a fortiori}, the image of $\gamma_{n}^c(X,\Lc)$ contains $2^{m-n}\Im\gamma_{m}^c(X,\Lc)$. By Example \ref{exe:link_exponent_size}, this yields the claims regarding exponents.
\end{proof}

In fact, in the situation of Lemma \ref{lem:descending_exponents}, we can sometimes be more precise regarding the image of the morphism $\gamma_n^c(X,\Lc)$. Let again $c$ and $m\geqslant n$ be non-negative integers and consider the exact sequence \[0\rightarrow\Ibf^m(\Lc)\rightarrow\Ibf^n(\Lc)\rightarrow\Ibf^n(\Lc)/\Ibf^m(\Lc)\rightarrow 0\] of abelian sheaves on $X$. Then there is an induced commutative ladder
\begin{center}
\begin{tikzcd}
\Hr^c(X,\Ibf^m(\Lc)) \arrow[r] \arrow[d,"\protect{\gamma_{m}^c(X,\Lc)}"] & \Hr^c(X,\Ibf^n(\Lc)) \arrow[r] \arrow[d,"\protect{\gamma_{n}^c(X,\Lc)}"] & \Hr^c(X,\Ibf^n(\Lc)/\Ibf^m(\Lc)) \arrow[r,"\partial"] \arrow[d,"\protect{\gamma_{n/m}^c(X,\Lc)}"] & \Hr^{c+1}(X,\Ibf^m(\Lc)) \arrow[d,"\protect{\gamma_{m}^{c+1}(X,\Lc)}"] \\
\Hr^c(X(\Rb),\Zb(L)) \arrow[r,swap,"2^{m-n}"] & \Hr^c(X(\Rb),\Zb(L)) \arrow[r,swap,"\rho(L)"] & \Hr^c(X(\Rb),\Zb/2^{m-n}(L)) \arrow[r] & \Hr^{c+1}(X(\Rb),\Zb(L))
\end{tikzcd}
\end{center}
of abelian groups with exact rows, where $\rho(L):\Hr^c(X(\Rb),\Zb(L))\rightarrow\Hr^c(X(\Rb),\Zb/2^{m-n}(L))$ is reduction mod $2^{m-n}$ of the coefficients.

\begin{lem}\label{lem:image_cycle_class_four_lemma}
If the morphism $\gamma_m^{c+1}(X,\Lc)$ is injective on $\Im\partial$ and the map $\gamma_m^c(X,\Lc)$ is surjective, then the image of $\gamma_{n}^c(X,\Lc)$ is the inverse image under $\rho(L)$ of $\Im\gamma_{n/m}^c(X,\Lc)$.
\end{lem}

\begin{proof}
Under these hypotheses, the usual four lemma applies and immediately yields the result.
\end{proof}

Let $d$ denote the dimension of $X$. By Theorem \ref{theo:twisted_signature_iso of sheaves}, the morphism $\gamma_{d+1}^c(X,\Lc):\Hr^c(X,\Ibf^{d+1}(\Lc))\rightarrow\Hr^c(X(\Rb),\Zb(L))$ is an isomorphism for every non-negative integer $c$. Thus we can apply Lemma \ref{lem:image_cycle_class_four_lemma} with $m=d+1$, with $n$ any integer such that $0\leqslant n\leqslant d$, and with every $c\geqslant 0$. In this situation, Lemma \ref{lem:image_cycle_class_four_lemma} tells us that the group $\Im\gamma_n^c(X,\Lc)$ is the inverse image of $\Im\gamma_{n/d+1}^c(X,\Lc)\subseteq\Hr^c(X(\Rb),\Zb/2^{d+1-n}(L))$ under the reduction mod $2^{d+1-n}$ map $\rho(L):\Hr^c(X(\Rb),\Zb(L))\rightarrow\Hr^c(X(\Rb),\Zb/2^{d+1-n}(L))$. In particular, since $0\in\Im\gamma_{n/d+1}^c(X,\Lc)$, the inverse image $2^{d+1-n}\Hr^c(X(\Rb),\Zb(L))$ of $\{0\}$ is contained in $\Im\gamma_n^c(X,\Lc)$. The following corollary is then immediate:

\begin{cor}\label{cor:upper_bound_for_exponent}
For every $0\leqslant n\leqslant d$ and every $c\geqslant 0$, the group $\Coker\gamma_n^c(X,\Lc)$ has exponent $2^{d+1-n}$.
\end{cor}

\subsection{Applications to classes of curves and surfaces}

Let again $X$ be a connected smooth $\Rb$-variety of dimension $d$ and let $\Lc$ be a line bundle on $X$, and set $L=\Lc(\Rb)$. In this subsection, we apply the results of the previous subsection to obtain exponents for the cokernel of the map $\gamma^c(X,\Lc):\Hr^c(X,\Ibf^c(\Lc))\rightarrow\Hr^c(X(\Rb),\Zb(L))$ in the cases $c=d-1$ (unconditionally on $X$) and $c=d-2$ (in a number of significant cases).

\begin{prop}\label{prop:cohomology_classes_of_curves}
The image of the morphism $\gamma^{d-1}(X,\Lc)$ is the inverse image of the subgroup $\Im\overline{\gamma}^{d-1}(X)$ of $\Hr^{d-1}(X(\Rb),\Zb/2)$ under the reduction mod $2$ homomorphism $\rho(L):\Hr^{d-1}(X(\Rb),\Zb(L))\rightarrow\Hr^{d-1}(X(\Rb),\Zb/2)$. In particular, the group $\Coker\gamma^{d-1}(X,\Lc)$ has exponent $2$.
\end{prop}

\begin{rema}
Recall that for every $n\geqslant 0$, modulo the natural isomorphism $\Hr^n(X,\overline{\Ibf}^n)\cong\CH^n(X)/2$, the morphism $\overline{\gamma}^n(X)$ coincides with the Borel--Haefliger cycle class map $\gamma_\BH^n(X):\CH^n(X)/2\rightarrow\Hr^n(X(\Rb),\Zb/2)$. Therefore in Proposition \ref{prop:cohomology_classes_of_curves}, the group $\Im\overline{\gamma}^{d-1}(X)$ is precisely the subgroup $\Hr_\alg^{d-1}(X(\Rb),\Zb/2)\subseteq\Hr^{d-1}(X(\Rb),\Zb/2)$ of algebraic classes mentioned in the introduction. Hence Proposition \ref{prop:cohomology_classes_of_curves} tells us that $\Im\gamma^{d-1}(X,\Lc)$ is precisely the subgroup of $\Hr^{d-1}(X(\Rb),\Zb(L))$ whose elements are those classes whose reduction mod $2$ is algebraic in the usual sense. This was also observed in \cite[Corollary 5.3]{hornbostelFewComputationsReal2024} which also contains the final statement of Proposition \ref{prop:cohomology_classes_of_curves}.
\end{rema}

\begin{proof}
Proposition \ref{prop:cohomology_classes_of_curves} is obvious if $X(\Rb)$ is empty (both groups are zero) so we may assume that $X(\Rb)\neq\emptyset$. We check that the hypotheses of Lemma \ref{lem:image_cycle_class_four_lemma} apply with $(m,n,c)=(d,d-1,d-1)$, namely for the inclusion $\Ibf^d(\Lc)\rightarrow\Ibf^{d-1}(\Lc)$ in degree $d-1$. It suffices to verify that the map $\gamma_d^{d-1}(X,\Lc):\Hr^{d-1}(X,\Ibf^d(\Lc))\rightarrow\Hr^{d-1}(X(\Rb),\Zb(L))$ is surjective, and that the map $\gamma^d(X,\Lc):\Hr^d(X,\Ibf^d(\Lc))\rightarrow\Hr^d(X(\Rb),\Zb(L))$ is injective. The injectivity statement is part of Theorem \ref{theo:conjecture_for_closed_points} since $X(\Rb)$ is non-empty by assumption. Similarly, the exact sequence (\ref{eq:exact_sequence_kernel_-1_i^d}) of sheaves determines an exact sequence \[\Hr^{d-1}(X,\Ibf^d(\Lc))\rightarrow\Hr^{d-1}(X,\Ibf^{d+1}(\Lc))\rightarrow\Hr^d(X,\overline{\Kbf})\] whose first map is identified by the isomorphism $\gamma_{d+1}^{d-1}(X,\Lc):\Hr^{d-1}(X,\Ibf^{d+1}(\Lc))\cong\Hr^{d-1}(X(\Rb),\Zb(L))$ with the map $\gamma_d^{d-1}(X,\Lc)$. The group $\Hr^d(X,\overline{\Kbf})$ vanishes by Proposition \ref{prop:injectivity_real_cycle_class_map_closed_points} again because $X(\Rb)$ is not empty. Consequently, the map $\gamma_d^{d-1}(X,\Lc)$ is surjective as required. Finally, since $0$ lies in the group $\Im\overline{\gamma}^{d-1}(X)$, the subgroup $\rho(L)^{-1}(\Im\overline{\gamma}^{d-1}(X))=\Im\gamma^{d-1}(X,\Lc)$ of $\Hr^{d-1}(X(\Rb),\Zb(L))$ contains $\Ker\rho(L)=2\Hr^{d-1}(X(\Rb),\Zb(L))$. This implies the final statement of Proposition \ref{prop:cohomology_classes_of_curves} and completes the proof.
\end{proof}

The following proposition, which provides a description of the cohomology group $\Hr^{d-1}(X,\Ibf^d(\Lc))$ in a number of meaningful cases, will also allow us to analyse the cycle class map $\gamma^{d-2}(X,\Lc)$.

\begin{prop}\label{prop:surjectivity_surfaces}
Assume that $X$ is not proper or that $X(\Rb)\neq\emptyset$, and that $\Hr_\et^{2d-1}(X\times_\Rb\Spec\Cb,\Zb/2)=0$. Then the exact sequence \emph{(\ref{eq:exact_sequence_kernel_-1_i^d})} of sheaves induces a split exact sequence 
\begin{equation}\label{eq:description_hd-1_Id}
0\rightarrow\Hr^{d-1}(X,\overline{\Kbf})\rightarrow\Hr^{d-1}(X,\Ibf^d(\Lc))\rightarrow\Hr^{d-1}(X,\Ibf^{d+1}(\Lc))\rightarrow 0
\end{equation}
of abelian groups.
\end{prop}

\begin{rema}
The proof of the above proposition (replacing $\Ibf$-cohomology by $\Hsc$-cohomology) is implicit in \cite{colliot-theleneZerocyclesCohomologyReal1996}, see in particular the (delicate) proof of \cite[Theorem 3.2 (e), Corollary 3.3 (b)]{colliot-theleneZerocyclesCohomologyReal1996}.
\end{rema}

\begin{proof}
In view of Lemma \ref{lem:kernel_as_torsion}, the ladder (\ref{eq:commutative_ladder_I_to_overline_I}) induces a commutative ladder with exact rows:
\begin{center}
\begin{tikzcd}
\Hr^{d-1}(X,\overline{\Kbf}) \arrow[r,"i"] \arrow[d,equal] & \Hr^{d-1}(X,\Ibf^d(\Lc)) \arrow[r] \arrow[d,"\alpha"] & \Hr^{d-1}(X,\Ibf^{d+1}(\Lc)) \arrow[r] \arrow[d] & \Hr^{d}(X,\overline{\Kbf}) \arrow[d,equal] \\
\Hr^{d-1}(X,\overline{\Kbf}) \arrow[r,swap,"\overline{i}"] & \Hr^{d-1}(X,\Hsc^d) \arrow[r] & \Hr^{d-1}(X,\Hsc^{d+1}) \arrow[r] & \Hr^{d}(X,\overline{\Kbf})
\end{tikzcd}
\end{center}
in which $\Hr^d(X,\overline{\Kbf})=0$ by Proposition \ref{prop:injectivity_real_cycle_class_map_closed_points}. This proves exactness of (\ref{eq:description_hd-1_Id}) at $\Hr^{d-1}(X,\Ibf^{d+1}(\Lc))$. Therefore to show the split exactness of the sequence (\ref{eq:description_hd-1_Id}), it suffices to show that the sequence 
\begin{equation}\label{eq:exact_sequence_hd-1_id-cohomology}
0\rightarrow\Hr^{d-1}(X,\overline{\Kbf})\xrightarrow{\overline{i}}\Hr^{d-1}(X,\Hsc^d)\rightarrow\Hr^{d-1}(X,\Hsc^{d+1})\rightarrow 0
\end{equation}
is split exact. Indeed, if $q:\Hr^{d-1}(X,\Hsc^d)\rightarrow\Hr^{d-1}(X,\overline{\Kbf})$ is a retraction of the map $\overline{i}$, so that $q\circ\overline{i}=\Id$, then the composite $\Hr^{d-1}(X,\Ibf^d(\Lc))\xrightarrow{\alpha}\Hr^{d-1}(X,\Hsc^d)\xrightarrow{q}\Hr^{d-1}(X,\overline{\Kbf})$ is a retraction of the map $i$.

Now let us prove that the sequence (\ref{eq:exact_sequence_hd-1_id-cohomology}) is split exact. Since the groups in this sequence are $\Zb/2$-vector spaces (by Remark \ref{rema:exponents_cohomology_groups}), if it is an exact sequence, then it automatically splits so we only have to prove that $\overline{i}$ is injective. This morphism fits in a cohomology exact sequence \[\Hr^{d-2}(X,\Hsc^d)\rightarrow\Hr^{d-2}(X,\Hsc^{d+1})\rightarrow\Hr^{d-1}(X,\overline{\Kbf})\xrightarrow{\overline{i}}\Hr^{d-1}(X,\Hsc^d)\] which implies that $\Ker\overline{i}$ is isomorphic to the cokernel of the homomorphism $\Hr^{d-2}(X,\Hsc^d)\rightarrow\Hr^{d-2}(X,\Hsc^{d+1})$. By \cite[Corollary 3.3 (b)]{colliot-theleneZerocyclesCohomologyReal1996}, this homomorphism is surjective under the hypotheses of Proposition \ref{prop:surjectivity_surfaces}, hence the map $\overline{i}$ is injective which completes the proof.
\end{proof}

\begin{cor}\label{cor:conjecture_for_surfaces}
Assume that the group $\Hr_\et^{2d-1}(X\times_\Rb\Spec\Cb,\Zb/2)$ vanishes. Then the morphism $\gamma_{d}^{d-2}(X,\Lc)$ is surjective. Consequently, the cokernel of the map $\gamma^{d-2}(X,\Lc)$ has exponent $4$.
\end{cor}

\begin{proof}
Note that if $X(\Rb)$ is empty, the conclusions of Corollary \ref{cor:conjecture_for_surfaces} are obvious so we may assume that $X(\Rb)\neq\emptyset$ which we now do. The morphism $\otimes\llangle -1\rrangle:\Ibf^d(\Lc)\rightarrow\Ibf^{d+1}(\Lc)$ induces a homomorphism $f:\Hr^{d-2}(X,\Ibf^d(\Lc))\rightarrow\Hr^{d-2}(X,\Ibf^{d+1}(\Lc))$ fitting in the cohomology exact sequence \[\Hr^{d-2}(X,\Ibf^d(\Lc))\xrightarrow{f}\Hr^{d-2}(X,\Ibf^{d+1}(\Lc))\rightarrow\Hr^{d-1}(X,\overline{\Kbf})\xrightarrow{i}\Hr^{d-1}(X,\Ibf^d(\Lc))\] induced by the short exact sequence (\ref{eq:exact_sequence_kernel_-1_i^d}) of sheaves. Consequently, the cokernel of the morphism $f$, which is identified by the isomorphism $\gamma_{d+1}^{d-2}(X,\Lc)$ with $\gamma_d^{d-2}(X,\Lc)$, is isomorphic to the kernel of $i$. Since $\Hr_\et^{2d-1}(X\times_\Rb\Spec\Cb,\Zb/2)=0$ and $X(\Rb)\neq\emptyset$, Proposition \ref{prop:surjectivity_surfaces} implies that $i$ is injective. Therefore $\Coker f\simeq\Coker\gamma_d^{d-2}(X,\Lc)$ vanishes so $\gamma_d^{d-2}(X,\Lc)$ is surjective. By Lemma \ref{lem:descending_exponents} with $(m,n,c)=(d,d-2,d-2)$, the group $\Coker\gamma^{d-2}(X,\Lc)$ then has exponent $2^{d-(d-2)}=4$ as required.
\end{proof}

\section{The quadratic real cycle class map for connected components}\label{section:connected_components}

Let again $X$ be a smooth real algebraic variety of dimension $d$ and let $\Lc\in\underline{\Pic}(X)$; we set $L=\Lc(\Rb)$. Using the method of the previous section, we can also analyse the case of connected components, namely the map \[\gamma^0(X,\Lc):\Hr^0(X,\Wbf(\Lc))\rightarrow\Hr^0(X(\Rb),\Zb(L)).\] Denote by $\Cc$ the set of connected components of $X(\Rb)$ and by $\Cc(L)$ the subset of $\Cc$ whose elements are those connected components $V$ of $X(\Rb)$ such that $L_{|V}$ is trivial. Note that $\Cc(L)=\Cc$ if, and only if, the locally free sheaf $L$ is globally free. Moreover, the group $\Hr^0(X(\Rb),\Zb)$ is the free abelian group on the set $\Cc$, and the abelian group $\Hr^0(X(\Rb),\Zb(L))$ is freely generated by $\Cc(L)$. This is a consequence of the following computation of sections of locally constant sheaves on $X(\Rb)$ (see in particular \cite[Proposition 2.58]{hornbostelRealCycleClass2021}): given $V\in\Cc$, we have \[
\Hr^0(V,\Zb(L))=\left\{
\begin{array}{ll}
\Zb & \text{if}\;L_{|V}\;\text{is\;trivial} \\
0   & \text{else}.
\end{array}\right.
\] We then have the following theorem of Krasnov and van Hamel (independently).

\begin{theo}[Krasnov, van Hamel]\label{theo:van_hamel}
The map $\overline{\gamma}_d^0(X):\Hr^0(X,\overline{\Ibf}^d)\rightarrow\Hr^0(X(\Rb),\Zb/2)$ is surjective.
\end{theo}

In fact, modulo the isomorphism $\Hsc^d\cong\overline{\Ibf}^d$ provided by the affirmation of the Milnor conjecture, the morphism $\overline{\gamma}_d^0(X)$ was constructed by Colliot-Thélène and Parimala in \cite[Section 2.1]{colliot-theleneRealComponentsAlgebraic1990}; denote it by $h_d^0$. By definition of $\Hsc^d$, there is a sheafification map $\Hr_\et^d(X,\Zb/2)\rightarrow\Hr^0(X,\Hsc^d)$ and the morphism $\beta:\Hr_\et^d(X,\Zb/2)\rightarrow\Hr^0(X(\Rb),\Zb/2)$ of \cite[Theorem 2.8]{hamelTorsionZerocyclesAbelJacobi2000} factors as
\begin{center}
\begin{tikzcd}
\Hr_\et^d(X,\Zb/2) \arrow[r] \arrow[rd,swap,"\beta"] & \Hr^0(X,\Hsc^d) \arrow[d,"h_d^0"] \\
                                       & \Hr^0(X(\Rb),\Zb/2)
\end{tikzcd}
\end{center}
(see in particular \cite[(7.19.1)]{scheidererRealEtaleCohomology1994} for a comparison of van Hamel's map and the signature map $h_d^0$ of Colliot-Thélène and Parimala). Then van Hamel observes in \cite[Theorem 2.8]{hamelTorsionZerocyclesAbelJacobi2000} that the stronger claim that the map $\beta:\Hr_\et^d(X,\Zb/2)\rightarrow\Hr^0(X(\Rb),\Zb/2)$ is surjective follows from \cite[Corollary 3.2]{krasnovEQUIVARIANTGROTHENDIECKCOHOMOLOGY1995a}, whose proof works for arbitrary $X$ smooth and connected although it is only stated for $X$ projective in Krasnov's work. Another argument is given in \cite[Remark 1.6 (ii)]{benoistIntegralHodgeConjecture2020}, that reads as follows. Every connected component $V$ of $X(\Rb)$ has an equivariant fundamental class $[V]_G$ in $\Hr_G^d(X(\Cb),\Zb/2)$ which can be regarded as an element $s_V$ of $\Hr_\et^d(X,\Zb/2)$ using Cox's isomorphism $\Hr_G^*(X(\Cb),\Zb/2)\simeq\Hr_\et^*(X,\Zb/2)$ \cite{coxEtaleHomotopyType1979}. The image of $s_V$ under $\beta$ is then precisely the generator of $\Hr^0(X(\Rb),\Zb/2)=\bigoplus_{W\in\mathcal{C}}\Zb/2\cdot[W]$ determined by $V\in\mathcal{C}$. In particular, the image of $\beta$, which is a subgroup of $\Hr^0(X(\Rb),\Zb/2)$, must then be equal to $\Hr^0(X(\Rb),\Zb/2)$.

\begin{cor}\label{cor:connected_components_Id_cohomology}
The map $\gamma_d^0(X,\Lc):\Hr^0(X,\Ibf^d(\Lc))\rightarrow\Hr^0(X(\Rb),\Zb(L))$ is surjective.
\end{cor}

\begin{proof}
Consider the following commutative ladder with exact rows:
\begin{center}
\begin{tikzcd}
\Hr^0(X,\Ibf^{d+1}(\Lc)) \arrow[r] \arrow[d,"\protect{\gamma_{d+1}^0(X,\Lc)}"] & \Hr^0(X,\Ibf^d(\Lc)) \arrow[r] \arrow[d,"\protect{\gamma_d^0(X,\Lc)}"] & \Hr^0(X,\overline{\Ibf}^d) \arrow[r] \arrow[d,"\protect{\overline{\gamma}_d^0(X)}"] & \Hr^1(X,\Ibf^{d+1}(\Lc)) \arrow[d,"\protect{\gamma_{d+1}^1(X,\Lc)}"] \\
\Hr^0(X(\Rb),\Zb(L)) \arrow[r,swap,"2"] & \Hr^0(X(\Rb),\Zb(L)) \arrow[r,swap,"\protect{\rho(L)}"] & \Hr^0(X(\Rb),\Zb/2) \arrow[r] & \Hr^1(X(\Rb),\Zb(L))
\end{tikzcd}
\end{center}
induced by the inclusion $\Ibf^{d+1}(\Lc)\rightarrow\Ibf^d(\Lc)$. The outer vertical maps $\gamma_{d+1}^0(X,\Lc)$ and $\gamma_{d+1}^1(X,\Lc)$ are isomorphisms by Theorem \ref{theo:twisted_signature_iso of sheaves}. By Lemma \ref{lem:image_cycle_class_four_lemma} with $(m,n,c)=(d+1,d,0)$, the image of $\gamma_d^0(X,\Lc)$ is then the subgroup $\rho(L)^{-1}(\Im\overline{\gamma}_d^0(X))$ of $\Hr^0(X(\Rb),\Zb(L))$. Since $\overline{\gamma}_d^0(X)$ is surjective by Theorem \ref{theo:van_hamel}, it follows that $\gamma_d^0(X,\Lc)$ is also surjective as required.
\end{proof}

\begin{prop}\label{prop:conjecture_untwisted_conn_comp}
The cokernel of $\gamma^0(X,\Lc):\Hr^0(X,\Wbf(\Lc))\rightarrow\Hr^0(X(\Rb),\Zb(L))$ has exponent $2^d$.
\end{prop}

\begin{proof}
The map $\gamma_d^0(X,\Lc)$ is surjective by Corollary \ref{cor:connected_components_Id_cohomology}. By Lemma \ref{lem:descending_exponents} with $(m,n,c)=(d,0,0)$, the group $\Coker\gamma^0(X,\Lc)$ then has exponent $2^{d-0}$.
\end{proof}

\begin{rema}
Let $V$ be a smooth real surface. According to Proposition \ref{prop:conjecture_untwisted_conn_comp}, the image of $\gamma^0(V)$ contains $4\Hr^0(V(\Rb),\Zb)$. By \cite[Corollary 7.2]{balmerGerstenWittSpectral2002}, the sheafification map $\Wr(V)\rightarrow\Wbf(V)$ is an isomorphism modulo which we may view $\gamma^0(V)$ as the global signature homomorphism $\Wr(V)\rightarrow\Hr^0(V(\Rb),\Zb)$. With the notations of \cite[Section 4]{karoubiWittGroupReal2020}, the image of this homomorphism is $S_a$ and by \cite[Remark 4.5.1]{karoubiWittGroupReal2020} and Proposition \ref{prop:conjecture_untwisted_conn_comp}, it agrees with the image $S_t$ of the topological signature defined there.
\end{rema}

We conclude by noting that over curves, we can improve Proposition \ref{prop:conjecture_untwisted_conn_comp} by computing the image of $\gamma^0(X,\Lc)$ using Proposition \ref{prop:cohomology_classes_of_curves}. Suppose until the end of the present section \ref{section:connected_components} that $d=1$, so that $X$ is a smooth real algebraic curve. Let $\rho(L):\Hr^0(X(\Rb),\Zb(L))\rightarrow\Hr^0(X(\Rb),\Zb/2)$ denote the reduction mod $2$ of the coefficients. As usual, we simply write $\rho$ if $L$ is trivial and we note that in general, the homomorphism $\rho(L)$ is the restriction of $\rho$ to $\Hr^0(X(\Rb),\Zb(L))\subseteq\Hr^0(X(\Rb),\Zb)$. Let $\Gamma(L)$ be the subgroup of $\Hr^0(X(\Rb),\Zb(L))$ defined as follows. If $L$ is trivial, then set $\Gamma(L)=\Gamma=2\Hr^0(X(\Rb),\Zb)+\Zb\cdot\varepsilon$ where $\varepsilon$ is the diagonal element defined as the tuple $(n_V)_{V\in\Cc}$ with $n_V=1$ for every $V\in\Cc$. If $L$ is non-trivial, then set $\Gamma(L)=2\Hr^0(X(\Rb),\Zb(L))$.

\begin{lem}\label{lem:maximal_signature_as_inverse_image}
The group $\Gamma(L)$ is the inverse image under the reduction map $\rho(L)$ of the subgroup $G=\Zb/2\cdot\overline{\varepsilon}$ of $\Hr^0(X(\Rb),\Zb/2)$ generated by $\overline{\varepsilon}=\rho(\varepsilon)$.
\end{lem}

\begin{proof}
Assume that $L$ is trivial. Since $G$ is generated by $\overline{\varepsilon}$ and $\rho(\varepsilon)=\overline{\varepsilon}$, it follows that $\rho^{-1}(G)=\Ker\rho+\Zb\cdot\varepsilon$, where $\Ker\rho=2\Hr^0(X(\Rb),\Zb)$. This completes the proof in this case.

Now assume that $L$ is non-trivial (observe that this assumption implies that $X(\Rb)$ is non-empty). Let $V\in\Cc\setminus\Cc(L)$ be a connected component of $X(\Rb)$ such that $L_{|V}$ is non-trivial. Projection onto the summand corresponding to $V$, whose generator we denote by $[V]$, induces a commutative square
\begin{center}
\begin{tikzcd}
\Hr^0(X(\Rb),\Zb) \arrow[d,swap,"\protect{\pi_V}"] \arrow[r,"\rho"] & \Hr^0(X(\Rb),\Zb/2) \arrow[d,"\protect{\overline{\pi}_V}"]\\
\Zb\cdot[V] \arrow[r,"\protect{\rho_{|\Zb\cdot[V]}}"]                     & \Zb/2\cdot[V]
\end{tikzcd}
\end{center}
We note that by choice of $V$, the group $\Hr^0(X(\Rb),\Zb(L))$ is contained in $\Ker\pi_V$ hence $\rho(\Hr^0(X(\Rb),\Zb(L)))=\Im\rho(L)$ is contained in $\Ker\overline{\pi}_V$. By definition of $\overline{\varepsilon}$, we have $\overline{\pi}_V(\overline{\varepsilon})=[V]$, thus $\overline{\varepsilon}$ does \emph{not} lie in $\Ker\overline{\pi}_V$. Consequently, the inclusion $\Im\rho(L)\subseteq\Ker\overline{\pi}_V$ implies that $\Im\rho(L)\cap G=\{0\}$. Therefore, to prove that $\Gamma(L)=\rho(L)^{-1}(G)$, it suffices to show the equality $\Gamma(L)=\rho(L)^{-1}(\{0\})$. The exact sequence \[\Hr^0(X(\Rb),\Zb(L))\xrightarrow{2}\Hr^0(X(\Rb),\Zb(L))\xrightarrow{\rho(L)}\Hr^0(X(\Rb),\Zb/2)\] implies that $\rho(L)^{-1}(\{0\})=2\Hr^0(X(\Rb),\Zb(L))$, hence $\rho(L)^{-1}(\{0\})=\Gamma(L)$ by definition of $\Gamma(L)$ as required.
\end{proof}

\begin{prop}\label{prop:image_signature_map}
We have an equality $\Im\gamma^0(X,\Lc)=\Gamma(L)$ of subgroups of $\Hr^0(X(\Rb),\Zb(L))$.
\end{prop}

\begin{proof}
By Proposition \ref{prop:cohomology_classes_of_curves} with $d=1$, the image of $\gamma^0(X,\Lc)$ is the group $\rho(L)^{-1}(\Im\overline{\gamma}^0(X))$. By Lemma \ref{lem:maximal_signature_as_inverse_image}, and using its notations, it now suffices to prove that $\Im\overline{\gamma}^0(X)=\Zb/2\cdot\overline{\varepsilon}$. Under the isomorphism $\Hr^0(X,\overline{\Wbf})\cong\CH^0(X)/2$, the map $\overline{\gamma}^0(X)$ is the Borel--Haefliger cycle class map $\gamma_\BH^0(X):\CH^0(X)/2\rightarrow\Hr^0(X(\Rb),\Zb/2)$. The image of this map is generated by the fundamental class of $X(\Rb)$ in $\Hr^0(X(\Rb),\Zb/2)$ which is precisely $\overline{\varepsilon}$.
\end{proof}

\section{A refined conjecture on the image of the quadratic real cycle class map}

As indicated in the introduction, it is not reasonable to expect the map $\gamma^c(X,\Lc)$ to be surjective in general as the following basic obstruction indicates.

\begin{exe}\label{exe:connected_components_of_curves}
Let $X$ be a connected smooth variety over $\Rb$ of dimension $\leqslant 3$. By \cite[Corollary 7.2]{balmerGerstenWittSpectral2002}, the map $\Wr(X)\rightarrow\Hr^0(X,\Wbf)$ induced by sheafification is then an isomorphism. It follows that $\gamma^0(X)$ and the global signature map $\sign:\Wr(X)\rightarrow\Hr^0(X(\Rb),\Zb)$ have the same image. We view the elements of $\Hr^0(X(\Rb),\Zb)$ as families of integers indexed by the connected components of $X(\Rb)$. If $b\in\Wr(X)$, then since the rank and signature of a symmetric bilinear form over $X$ are equal mod $2$, the integers appearing in the tuple $\sign(b)$ are either all even or all odd. This condition defines a subgroup $\Gamma$ of $\Hr^0(X(\Rb),\Zb)$ which is strict if $X(\Rb)$ has at least two connected components. In fact, \cite[Theorem (10.4) iv)]{knebuschAlgebraicCurvesReal1976a} states that the image of the signature map is precisely $\Gamma$ if the dimension $d$ of $X$ is $1$. This is precisely Proposition \ref{prop:image_signature_map} with trivial twisting line bundle in view of the isomorphism $\Wr(X)\xrightarrow{\cong}\Hr^0(X,\Wbf)$; since \cite[Corollary 7.2]{balmerGerstenWittSpectral2002} also yields an isomorphism $\Wr(X,\Lc)\cong\Hr^0(X,\Wbf(\Lc))$, we may regard Proposition \ref{prop:image_signature_map} as a generalisation of Knebusch's computation to arbitrary twisting line bundle.
\end{exe}

Let $X$ be a smooth $\Rb$-variety of dimension $d$ and let $\Lc$ be a line bundle on $X$; set $L=\Lc(\Rb)$. Analytic techniques are not available on the smooth manifold $X(\Rb)$. As a result, there is no natural candidate subgroup of $\Hr^c(X(\Rb),\Zb(L))$ through which $\gamma^c(X,\Lc)$ could factor; in this regard, we note that the inverse image of $\Im\gamma_{c/d+1}^c(X,\Lc)$ under the reduction homomorphism $\Hr^c(X(\Rb),\Zb(L))\rightarrow\Hr^c(X(\Rb),\Zb/2^{d+1-c}(L))$ that we mentioned following Lemma \ref{lem:image_cycle_class_four_lemma} is just as mysterious as $\Im\gamma^c(X,\Lc)$ in general. To estimate the defect of the quadratic real cycle class map to be surjective, we will instead study the exponents of $\Coker\gamma^c(X,\Lc)$ or, equivalently, the integers $e$ such that $\Im\gamma^c(X,\Lc)$ contains $e\Hr^c(X(\Rb)\Zb(L))$---the morphism $\gamma^c(X,\Lc)$ is closer to being surjective the smaller $e$ is. Jacobson's theorem \ref{theo:twisted_signature_iso of sheaves} provides the following result.

\begin{prop}\label{prop:kernel_cokernel_real_cycle_class_map_torsion}
Let $(n,c)$ be a pair of integers such that $n\leqslant d$. Then $\Ker\gamma_{n}^c(X,\Lc)$ has exponent $2^{2(d+1-n)}$.
\end{prop}

\begin{proof}
Let $\Kbf(\Lc)$ (resp. $\Qbf(\Lc)$, resp. $\Jbf(\Lc)$) denote the kernel (resp. cokernel, resp. image) of the multiplication by $\llangle -1\rrangle^{\otimes(d+1-n)}$ map from $\Ibf^n(\Lc)$ to $\Ibf^{d+1}(\Lc)$. Recall that $\llangle -1\rrangle=2$ in the Witt ring hence this map is also multiplication by $2^{d+1-n}$ (more precisely, it is the factored map of the multiplication map $\cdot 2^{d+1-n}:\Ibf^n(\Lc)\rightarrow\Ibf^n(\Lc)$ by the subsheaf $\Ibf^{d+1}(\Lc)$ of $\Ibf^n(\Lc)$).

We claim that the sheaves $\Kbf(\Lc)$ and $\Qbf(\Lc)$ have exponent $2^{d+1-n}$. Indeed, for the kernel $\Kbf(\Lc)$, this is true because the morphism $\otimes\llangle -1\rrangle^{\otimes(d+1-n)}$ is multiplication by $2^{d+1-n}$ up to composition with a monomorphism. On the other hand, since $\Ibf^{d+1}(\Lc)$ is a subsheaf of $\Ibf^n(\Lc)$, there is an inclusion $2^{d+1-n}\Ibf^{d+1}(\Lc)\subseteq 2^{d+1-n}\Ibf^n(\Lc)=\Jbf(\Lc)$ of subsheaves of $\Ibf^{d+1}(\Lc)$. It follows that $\Qbf(\Lc)=\Ibf^{d+1}(\Lc)/\Jbf(\Lc)$ is a quotient of $\Ibf^{d+1}(\Lc)/2^{d+1-n}\Ibf^{d+1}(\Lc)$, which has exponent $2^{d+1-n}$, so the sheaf $\Qbf(\Lc)$ also has exponent $2^{d+1-n}$. We conclude that the cohomology groups of $X$ with coefficients in $\Kbf(\Lc)$ and $\Qbf(\Lc)$ have exponent $2^{d+1-n}$.

Up to the isomorphism $\gamma_{d+1}^c(X,\Lc)$, the morphism $\gamma_{n}^c(X,\Lc)$ coincides with the map $f:\Hr^c(X,\Ibf^n(\Lc))\rightarrow\Hr^c(X,\Ibf^{d+1}(\Lc))$ induced in cohomology by the homomorphism $\otimes\llangle -1\rrangle^{\otimes(d+1-n)}$ of sheaves, hence it suffices to prove that $\Ker f$ has exponent $2^{2(d+1-n)}$. The morphism $f$ factors through cohomology with coefficients in the image $\Jbf(\Lc)$ as \[\Hr^c(X,\Ibf^n(\Lc))\xrightarrow{u}\Hr^c(X,\Jbf(\Lc))\xrightarrow{u'}\Hr^c(X,\Ibf^{d+1}(\Lc))\] where $u$ and $u'$ are induced by the morphisms $\Ibf^n(\Lc)\rightarrow\Jbf(\Lc)$ and $\Jbf(\Lc)\rightarrow\Ibf^{d+1}(\Lc)$ of sheaves respectively, hence $\Ker f$ is an extension of $\Ker u$ by $\Ker u'$. It now suffices to show that $\Ker u$ and $\Ker u'$ have exponent $2^{d+1-n}$. The exact sequence \[0\rightarrow\Kbf(\Lc)\rightarrow\Ibf^n(\Lc)\rightarrow\Jbf(\Lc)\rightarrow 0\] of sheaves on $X$ induces an exact sequence \[\Hr^c(X,\Kbf(\Lc))\rightarrow\Hr^c(X,\Ibf^n(\Lc))\xrightarrow{u}\Hr^c(X,\Jbf(\Lc))\] of abelian groups. Thus $\Ker u$ is a quotient of $\Hr^c(X,\Kbf(\Lc))$ which has exponent $2^{d+1-n}$ as already observed so that $\Ker u$ has exponent $2^{d+1-n}$. Similarly, the exact sequence \[0\rightarrow\Jbf(\Lc)\rightarrow\Ibf^{d+1}(\Lc)\rightarrow\Qbf(\Lc)\rightarrow 0\] induces an exact sequence \[\Hr^{c-1}(X,\Qbf(\Lc))\rightarrow\Hr^c(X,\Jbf(\Lc))\xrightarrow{u'}\Hr^c(X,\Ibf^{d+1}(\Lc))\] of abelian groups, where $\Hr^{c-1}(X,\Qbf(\Lc))$ has exponent $2^{d+1-n}$ as previously observed. We deduce as before that $\Ker u'$ has exponent $2^{d+1-n}$. This completes the proof.
\end{proof}

This proposition, together with Corollary \ref{cor:upper_bound_for_exponent}, are quantitative versions of the following result of Jacobson.

\begin{cor}[\protect{\cite[Corollary 8.9 (2)]{jacobsonRealCohomologyPowers2017}}]\label{cor:iso_after_inverting_2}
For every pair $(n,c)$ of integers, the map \[\gamma_{n}^c(X,\Lc)\otimes\Zb[1/2]:\Hr^c(X,\Ibf^n(\Lc))\otimes\Zb[1/2]\rightarrow\Hr^c(X(\Rb),\Zb[1/2](L))\] obtained by inverting $2$ is an isomorphism.
\end{cor}

\begin{rema}
This stands in sharp contrast with the status of the rational Hodge conjecture whose validity is unknown: we are really interested in integral statements.
\end{rema}

The exponents obtained in Corollary \ref{cor:upper_bound_for_exponent} and Proposition \ref{prop:kernel_cokernel_real_cycle_class_map_torsion} are not optimal. For instance, if $n=c=0$ and $d=1$, it follows from Knebusch's work \cite{knebuschAlgebraicCurvesReal1976a} cited in Example \ref{exe:connected_components_of_curves} that $\Ker\gamma^0(X)$ and $\Coker\gamma^0(X)$ have exponent $2$, whereas Proposition \ref{prop:kernel_cokernel_real_cycle_class_map_torsion} gives the exponent $2^{2(1+1-0)}=16$ for the kernel and Corollary \ref{cor:upper_bound_for_exponent} gives the exponent $2^{1+1-0}=4$ for the cokernel.

We contend that the exponent of Corollary \ref{cor:upper_bound_for_exponent} for the cokernel can be improved, but only slightly. To state our conjecture precisely, consider the following assertions, parametrised by triples $(X,\Lc,c)$ where $X$ is a smooth $\Rb$-variety of dimension $d$, the letter $\Lc$ denotes a line bundle on $X$ inducing a locally free module $L=\Lc(\Rb)$ on $X(\Rb)$ and $c$ is an integer such that $0\leqslant c\leqslant d$.
\begin{quote}
$P(X,\Lc,c)$: \emph{The cokernel of the map $\gamma^c(X,\Lc):\Hr^c(X,\Ibf^c(\Lc))\rightarrow\Hr^c(X(\Rb),\Zb(L))$ has exponent $2^{d-c}$.}
\end{quote}
Note that by Lemma \ref{lem:chow_witt_i_cohomology_cycle_map_same_image}, Assertion $P(X,\Lc,c)$ is equivalent to
\begin{quote}
$\widetilde{P}(X,\Lc,c)$: \emph{The cokernel of the map $\widetilde{\gamma}^c(X,\Lc):\widetilde{\CH}^c(X,\Lc)\rightarrow\Hr^c(X(\Rb),\Zb(L))$ has exponent $2^{d-c}$.}
\end{quote}
For brevity, we only consider the property $P(X,\Lc,c)$. We can now state the refined conjecture mentioned in the introduction:

\begin{con}\label{con:refined_hodge_conjecture}
For every triple $(X,\Lc,c)$ as above, Assertion $P(X,\Lc,c)$ holds. Moreover, the exponent $2^{d-c}$ predicted by Assertion $P(X,\Lc,c)$ is the best possible in the following sense: for every pair $(c,d)$ of non-negative integers such that $c<d$, there exists a smooth variety $X$ over $\Rb$ of dimension $d$ and a line bundle $\Lc$ on $X$ such that $\Coker\gamma^c(X,\Lc)$ does not have exponent $2^{d-c-1}$.
\end{con}

We observe that a number of positive results towards this conjecture have been collected in the previous sections:

\begin{theo}\label{theo:cases_of_conjecture}
Let $X$ be a connected smooth $\Rb$-variety of dimension $d$ and let $\Lc$ be a line bundle on $X$. Then Assertions $P(X,\Lc,i)$ holds for $i\in\{d,d-1,0\}$. Moreover $P(X,\Lc,d-2)$ also holds if the étale cohomology group $\Hr_\et^{2d-1}(X\times_\Rb\Spec\Cb,\Zb/2)$ vanishes.
\end{theo}

\begin{proof}
In codimension $d$, we have to prove that $\Coker\gamma^d(X,\Lc)$ has exponent $2^{d-d}=1$, namely that this cokernel vanishes, in other words that $\gamma^d(X,\Lc)$ is surjective: this is Theorem \ref{theo:conjecture_for_closed_points}. In codimension $d-1$, we have to show that $\Coker\gamma^{d-1}(X,\Lc)$ has exponent $2^{d-(d-1)}=2$: this is the final statement of Proposition \ref{prop:cohomology_classes_of_curves}. In codimension $d-2$, we have to prove that $\Coker\gamma^{d-2}(X,\Lc)$ has exponent $2^{d-(d-2)}=4$ if $\Hr_{\et}^{2d-1}(X\times_\Rb\Spec\Cb,\Zb/2)$ vanishes: this is Corollary \ref{cor:conjecture_for_surfaces}. In codimension $0$, it remains to establish that $\Coker\gamma^0(X,\Lc)$ has exponent $2^d$: this is Proposition \ref{prop:conjecture_untwisted_conn_comp}.
\end{proof}

\begin{rema}
We refer the reader to \cite[Examples 4.3.1]{colliot-theleneZerocyclesCohomologyReal1996} for examples of satisfaction of the étale cohomological vanishing condition of Theorem \ref{theo:cases_of_conjecture}. We mention here that it is satisfied for affine varieties, for cohomological dimension reasons, and for proper varieties $X$ such that $X(\Cb)$ is simply connected since, by Poincaré duality, the group $\Hr_\et^{2d-1}(X\times_\Rb\Spec\Cb,\Zb/2)$ is then isomorphic to $\Hr_\et^1(X\times_\Rb\Spec\Cb,\Zb/2)$ which vanishes under the simple connectedness assumption.
\end{rema}

We now enumerate cases covered by Theorem \ref{theo:cases_of_conjecture} and comment on the optimality of exponents.

\begin{exe}
Let $X$ be a smooth real curve and let $\Lc$ be a line bundle on $X$. Then by Theorem \ref{theo:cases_of_conjecture}, Assertion $P(X,\Lc,c)$ is satisfied for all $c$ since the set of possible codimensions is $\{1,0\}=\{d,0\}$. Moreover, the exponents predicted by Conjecture \ref{con:refined_hodge_conjecture} are optimal. Indeed, we only have to check that there exists a smooth real curve $X$ such that the map $\gamma^0(X):\Wbf(X)\rightarrow\Hr^0(X(\Rb),\Zb)$ is not surjective. As explained in Example \ref{exe:connected_components_of_curves}, this map can be identified with the global signature morphism $\sign:\Wr(X)\rightarrow\Hr^0(X(\Rb),\Zb)$ whose image is strict if $X(\Rb)$ is not connected, for example if $X=\Abb^1\setminus 0$.
\end{exe}

\begin{exe}\label{exe:conj_for_surfaces}
Let $X$ be a smooth real surface and let $\Lc$ be a line bundle on $X$. Then again, by Theorem \ref{theo:cases_of_conjecture}, Assertion $P(X,\Lc,c)$ is satisfied for all $c$ since the set of possible codimensions is $\{2,1,0\}=\{d,d-1,0\}$. Note that the exponent $2^2=4$ is optimal for surfaces in codimension $0$: there exist smooth surfaces $X$ such that $\Im\gamma^0(X)$ does not contain $2\Hr^0(X(\Rb),\Zb)$. For example, according to \cite[Proposition 5.1]{monnierUnramifiedCohomologyQuadratic2000}, this is the case if $X$ is projective, geometrically connected and rational (that is, the complex variety $X\times_\Rb\Spec\Cb$ is birational to $\mathbb{P}_\Cb^2$), and $X(\Rb)$ has at least two connected components.
\end{exe}

\begin{exe}
Let $X$ be a smooth real $3$-fold and let $\Lc$ be a line bundle on $X$. Then Assertion $P(X,\Lc,c)$ is satisfied for every $c\in\{d=3,d-1=2,d=0\}$, and in codimension $c=1$ if $\Hr_\et^5(X\times_\Rb\Spec\Cb,\Zb/2)$ vanishes, for example if $X$ is affine or if $X$ is proper and $X(\Cb)$ is simply connected.
\end{exe}

\begin{exe}\label{exe:optimality_dimension_1}
Let $X=\Abb^d\setminus 0$ with $d\geqslant 2$. Then since $X$ is an open subscheme of $\Abb^d$, the restriction map $\CH^{d-1}(\Abb^d)\rightarrow\CH^{d-1}(X)$ is surjective and thus $\CH^{d-1}(X)$ vanishes since $\CH^{d-1}(\Abb^d)=0$ by $\Abb^1$-invariance of Chow groups (recall that $d-1\geqslant 1$). It follows that the homomorphism $\overline{\gamma}^{d-1}(X):\CH^{d-1}(X)\rightarrow\Hr^{d-1}(X(\Rb),\Zb/2)$ is the zero map, thus by Proposition \ref{prop:cohomology_classes_of_curves}, the image of $\gamma^{d-1}(X)$ is given by \[\rho^{-1}(\{0\})=\Ker\rho=2\Hr^{d-1}(X(\Rb),\Zb)\] where $\rho:\Hr^{d-1}(X(\Rb),\Zb)\rightarrow\Hr^{d-1}(X(\Rb),\Zb/2)$ is the reduction mod $2$ homomorphism. But $X(\Rb)$ is homotopic to the $(d-1)$-sphere $\Sr^{d-1}$ hence $\Hr^{d-1}(X(\Rb),\Zb)\simeq\Zb$ and thus $2\Hr^{d-1}(X(\Rb),\Zb)\neq\Hr^{d-1}(X(\Rb),\Zb)$. It follows that $\Coker\gamma^{d-1}(X)=\Hr^{d-1}(X(\Rb),\Zb)/2\Hr^{d-1}(X(\Rb),\Zb)$ is non-zero, that is, it does not have exponent $1$. Thus $2=2^{d-(d-1)}$ is the best possible exponent in codimension $d-1$. Combined with Example \ref{exe:conj_for_surfaces}, we then see that the exponents predicted by Conjecture \ref{con:refined_hodge_conjecture} are optimal for surfaces.
\end{exe}

\begin{rema}
Let $X$ be a smooth $\Rb$-variety of dimension $d$, let $\Lc$ be a line bundle on $X$ and let $c$ be an integer such that $0\leqslant c\leqslant d$; we set $L=\Lc(\Rb)$. Although analysis does not appear to provide a clear candidate subgroup containing $\Im\gamma^c(X,\Lc)$, one could instead take inspiration from the topological conditions used in \cite{benoistIntegralHodgeConjecture2020} based on Steenrod squares. Indeed, note that the twisted Steenrod squares $\Sq_L:\Hr^*(X(\Rb),\Zb/2)\rightarrow\Hr^{*+1}(X(\Rb),\Zb(L))\rightarrow\Hr^{*+1}(X(\Rb),\Zb/2)$ provide the differentials of the first page of the Bockstein spectral sequence associated with the filtration \[\cdots\rightarrow\Zb(L)\xrightarrow{2}\Zb(L)\rightarrow\cdots\rightarrow\Zb(L)\xrightarrow{2}\Zb(L)\] of $\Zb(L)$. On the other hand, one can consider the Pardon spectral sequence studied in \cite{pardonFILTEREDGESTENWITTRESOLUTION}, which is the spectral sequence associated with the filtered sheaf \[\cdots\rightarrow\Ibf^{n+1}(\Lc)\rightarrow\Ibf^n(\Lc)\rightarrow\cdots\rightarrow\Ibf(\Lc)\rightarrow\Wbf(\Lc).\] We note that Jacobson's signature homomorphisms induce a morphism $\Wbf(\Lc)\rightarrow\iota_*\Zb(\Lc)$ of filtered sheaves (where $\iota:X(\Rb)\hookrightarrow X$ is the inclusion as usual) and thus a morphism of spectral sequences from Pardon's construction to the Bockstein spectral sequence. Moreover, the classes in the cohomology of $\Ibf^n(\Lc)$ coming from higher terms of the filtration can be identified by vanishing conditions for differentials in higher pages, and the image of these classes under the quadratic real cycle class map must then be killed by the corresponding differentials. One can therefore envision to propose a candidate subgroup for the image of $\gamma^c(X,\Lc)$ in terms of these vanishing conditions, which could incidentally shed light on Conjecture \ref{con:refined_hodge_conjecture}. We hope to pursue this approach in future work. 
\end{rema}

\pagestyle{plain}

\printbibliography

\end{document}